\documentclass[12pt,a4paper]{article}

\usepackage[latin1]{inputenc}
\usepackage{hyperref}
\usepackage{amsmath,amssymb,amscd,amsthm,amsfonts}
\usepackage{calrsfs}
\usepackage[all,cmtip]{xy}
\usepackage{rotating}
\usepackage{makeidx}
\usepackage{setspace,graphicx}
\usepackage{color}
\usepackage{float}
\usepackage{amsxtra}
\usepackage{latexsym}
\usepackage{graphics}
\usepackage{mathrsfs}
\theoremstyle{definition}
\newtheorem{thm}{Theorem}[section]
\newtheorem{defi}[thm]{Definition}

\newtheorem{prop}[thm]{Proposition}
\newtheorem{lemma}[thm]{Lemma}
\newtheorem{cor}[thm]{Corollary}

\newtheorem{ex}{Example}

\DeclareMathAlphabet{\pazocal}{OMS}{zplm}{m}{n}

\newcommand{\RP}{\R\mathbb{P}}
\newcommand{\Q}{\mathbb Q}
\newcommand{\R}{\mathbb R}
\newcommand{\Z}{\mathbb Z}
\newcommand{\area}{area\,}

\newcommand{\diff}{differentiable\;}
\newcommand{\graf}{graph\;}
\newcommand{\fn}{function\,}

\newcommand{\param}{parametrisation\,}
\newcommand{\rank}{\textrm{rank}\,}

\newcommand{\ssi}{if\;and\;only\;if\;}
\newcommand{\asdir}{asymptotic\;direction\;}

\newcommand{\spp}{special\;parabolic\;point\;}
\newcommand{\spps}{special\;parabolic\;points\;}

\newcommand{\Supp}{\textrm{Supp}\,}
\newcommand{\Tub}{\mathrm{Tub}}

\DeclareMathOperator{\Gdef}{\Gamma_f}

\author{ Fuensanta Aroca, Angelito Camacho-Calder\'on,
	Mirna G\'omez-Morales.\\Unidad Cuernavaca del Instituto de Matem\'aticas,\\ Universidad Nacional Aut\'onoma de M\'exico.\\UMI - ``Laboratorio Solomon Lefschetz" CNRS}

\title{Transversal special parabolic points in the graph of a polynomial obtained under Viro's patchworking \thanks{Research partially supported by PAPIIT-UNAM IN108216, ECOS M14M03, LAISLA, CONACYT (Mexico) grants
		225387-292689 and 224855-291053.} }

\begin{document}

\maketitle

\begin{abstract}
	In this article we focus on the study of special parabolic points in surfaces arising as graphs of polynomials, we give a theorem of Viro's patchworking type to build families of real polynomials in two variables with a prescribed number of special parabolic points in their graphs. We use this result to build a family of degree $d$ real polynomials in two variables with $(d-4)(2d-9)$ special parabolic points in its graph. This brings the number of special parabolic points closer to the upper bound of $(d-2)(5d-12)$ when $d\geq 13$, which is the best known up until now.
\end{abstract}

\section{Introduction}
Points in a surface immersed in a $3$-dimensional affine space are classified in terms of the contact order of their tangent lines to the surface. On generic surfaces, parabolic points appear along a curve which separates the hiperbolic domain from the elliptic domain and, among parabolic points, there are points where the highest contact order is reached in the direction of its only asymptotic line. These points are called \textit{special parabolic points} or \textit{Gaussian cusps}. 

Finding the number of special parabolic points in the graph of a generic polynomial of degree $d$, has been of special interest for the last century. For example, in \cite{Kulikov1983}, an upper bound of $2d(d-2)(11d-24)$ special parabolic points in generic algebraic surfaces of degree $d$ in $\RP^3$ is given. In \cite{OrtizGeomdesSurfaces}, A. Ortiz-Rodr\'iguez builds a family of polynomials whose graphs describe generic surfaces with $d(d-2)$ special parabolic points. And in \cite{OrtizHdezAffinegeom}, together with Hern\'andez-Mart\'inez and S\'anchez-Bringas, she proves that there are at most $(d-2)(5d-12)$ special parabolic points in the graph of a polynomial of degree $d$.
 


Viro's patchworking was introduced in the late seventies \cite{Viro1984} as a technique to glue simple algebraic curves in order to construct real algebraic non-singular curves with prescribed topology. Details on this technique can be found for example in \cite{ViroPatchworking}. Among its many applications, it has been used by E. Brugall\'e and B. Bertrand to construct examples of real algebraic hypersurfaces in the projective plane with $(d-4)^2$ compact connected components in their parabolic curves \cite{BertrandErwan}, and by E. Brugall\'e and L. L\'opez de Medrano to construct examples of real algebraic curves in the projective plane with the maximum number of real inflection points \cite{brugalle2011inflection}. 

In this article, we glue simple graphs in order to build a new graph with a prescribed number of special parabolic points in it. Our main result is a theorem of Viro's patchworking type: 

\noindent
\textbf{Theorem \ref{Viro'sThmfortspps}}(Viro's Theorem for transversal special parabolic points)
	Let $\Delta \subset \R^2$ be a polyhedron with vertices in $\Z^2$ and let $\tau$ be the convex polyhedral subdivision of $\Delta$ induced by $\lambda :\Delta \to \R_{\geq 0}$. Let $f \in\R[x,y]$ be a polynomial with non-singular Hessian curve and support in $\Delta$. If $f_t$ is the patchworking polynomial of $f$ induced by $\lambda$, then there exists $\delta>0$ such that for $0<|t|<\delta$, there is an inclusion 
	\begin{equation*}
		\varphi_t:TSPP(f,\tau)^{*}\hookrightarrow TSPP(f_t)^{*}.
	\end{equation*}Here 
	\begin{equation*}
 		TSPP(f,\tau)^{*}:=\bigcup_{E\in \tau}\{(E,p); p\in TSSP(f|_E)^{*} \},
	\end{equation*} where for $E\in \tau$, $TSPP(f|_{E})^{*}$ denotes the set of transversal special parabolic points in the graph of the restriction map $f|_E$ that lie in $\pi^{-1}({\R^{*}}^2)$.

We use this theorem to disprove a conjecture that first appeared in 2002 in Ortiz's Phd dissertation \cite{TesisAdriana}, which was written under Arnold's supervision. In \cite{CounterexAdriana}, Hern\'andez-Mart\'inez, A. Ortiz. and F. S\'anchez-Bringas, give a degree 4 polynomial with 2 \spps above the bound of $d(d-2)$ given by A. Ortiz. Since $d(d-2)+2\leq (d-4)(2d-9)$ if $d\geq 13$, our theorem bring us closer to the bound $(d-2)(5d-12)$ special parabolic points given in \cite{OrtizHdezAffinegeom} in this case.

In Section \ref{sec:Classofpoints}, we give preliminaries on the classification of points in a surface and recall known results to characterise special parabolic points in the \graf of a \fn $f$ as the zero set of three polynomials $H_f$, $E_{1,f}$ and $E_{2,f}$ defined in terms of $f$. In Section \ref{sec:Viro'spatchworking}, we recall known results on convex triangulations and describe Viro's patchworking technique.
In Sections \ref{sec:Deftheory}, and \ref{sec:SPPsunderperturbation}, we describe how the variety defined by $H_f$, $E_{1,f}$ and $E_{2,f}$ behaves under the one-parameter perturbation $f+t\,g_t$ given in terms of polynomials $g_t\in\R[t][x,y]$; and in Sections \ref{sec:Sppsofhomotheticpols0} and \ref{sec:Sppsofhomotheticpols}, we analyse how the number of transversal special parabolic points is preserved under quasihomotheties of the form $(x,y)\mapsto (t^{\alpha}x,t^{\beta}y)$ for $t\neq 0$.

In Section \ref{sec:Viro'sthmforspps}, we state our main result to describe the behaviour of transversal special parabolic points under Viro's patchworking. Lastly, in Section \ref{sec:Application}, we use Corollary \ref{Cor:lemaViro} to build a one-parameter family $f_t\in\R[t][x,y]$ of polynomials of degree $d$ with at least $(d-4)(2d-9)$
transversal special parabolic points in their graphs for sufficiently small values of $t$. 

The authors would like to thank Erwan Brugall{\'e} and Adriana Ortiz-Rodr\'iguez for their seminars and for valuable discussions on the subject of real surfaces. In particular, to A. Ortiz-Rodr\'iguez for the proof of Proposition \ref{PPE=H,E1,E2=0}. The second author would like to thank Luc\'ia L\'opez de Medrano, for answering several questions on the subject. 

\section{Classification of points in a surface}\label{sec:Classofpoints}

\begin{defi}
	Let $S\subset \R^3$ be a surface defined by the vanishing set of a differentiable function $F: \R^3 \rightarrow \R$, that is, $S=\{(x,y,z)\in\R^3; F(x,y,z)=0\}$. Take $p\in S$ and let $l: \R \to \R^3$ be the linear \param of a line with $p=l(t_0)$. The line $l$ has \textit{contact order} $k\in \mathbb{N}$ with $S$ at $p$ \ssi the partial derivatives satisfy
		\begin{equation*}
			(F\circ l)^{(m)} (t_0)=0, \; \; \mbox{for} \; \; m=0,\ldots,k-1;		\quad \mbox{and} \quad
			(F\circ l)^{(k)}(t_0)\neq 0.
		\end{equation*}	
\end{defi}
Tangent lines to a point in a regular surface have contact order $k\geq 2$. Salmon G. \cite{Salmon1865} used this property to classify the points in a surface according to the following criteria. 
\begin{defi}
Let $p$ be a point in the regular surface $S\subset \R^3$. A line with contact order $k\geq 3$ at $p\in S$ is called an \textit{asymptotic direction}. A point $p\in S$ is called \vspace{-0.3cm}
	\begin{enumerate}\setlength\itemsep{-0.3em}
		\item[1)] \textit{elliptic} if all tangent lines to $S$ at $p$ have contact order equal to two;
		\item[2)] \textit{hyperbolic} if it has exactly two asymptotic directions; or 
		\item[3)] \textit{parabolic} if it has either one or more than two asymptotic directions. A parabolic point $p$ is also called \vspace{-0.4cm}
		\begin{itemize}\setlength\itemsep{-0.3em}
		\item[a)] \textit{generic} if it has only one \asdir $l$ and the contact order of $l$ at $p$ is 3;
		\item[b)] \textit{special} if it has only one \asdir $l$ with contact order $k\geq 4$; or 
		\item[c)] \textit{degenerate} if it has more than two asymptotic directions.
		\end{itemize}
	\end{enumerate}\vspace{-0.4cm}
	The set of parabolic points in a non-degenerate surface $S\subset \R^3$ forms a curve called the \textit{parabolic curve} of $S$.
\end{defi}

Let $S\subset \R^3$ be locally expressed as the \graf 
\begin{equation*}
	\Gdef=\{(x,y,z)\in \R^{3}| f(x,y)=z\}
\end{equation*}
of a \diff \fn $f:\R^2\rightarrow \R$. We will consider from now on the standard projection $\pi:\R^3\rightarrow \R^2$ on the $xy$-plane and we will denote by $SPP(f)$ the set of special parabolic points in $\Gdef$, and by $SPP(f)^{*}$ the points in $SPP(f)$ that lie in $\pi^{-1}({\R^{*}}^2)$.

Hereafter, we will denote the vanishing set of a function $f$ as $V(f)$.
\begin{defi}
	Let $f: \R^2 \rightarrow \R$ be a \diff \fn. We will refer to the curve $V(H_f)$, defined by the \textit{Hessian} $H_f(x,y):=f_{xx}f_{yy}-f_{xy}^2$ of $f$, as the \textit{Hessian curve} of $f$.
\end{defi}\vspace{-0.5cm}
Note that the Hessian $H_f$ of $f$ is the determinant of its Hessian matrix $Hess(f)=\begin{pmatrix}f_{xx} & f_{xy} \\f_{yx} & f_{yy}\end{pmatrix}$. When the graph of a function is a non-degenerate surface, the Hessian of the function, along with the following three functions, plays an important role in finding special parabolic points.

\begin{defi}\label{Ei's} 
	Let $f: \R^2 \rightarrow \R$ be a \diff \fn. We consider the functions $C_{f}, E_{1,f}$ and $E_{2,f}$ given 
	by \vspace{-0.4cm}	
		\begin{itemize}\setlength\itemsep{-0.5em}
			\item [i)] 	 $C_{f}(x,y):=\begin{pmatrix}
				-(H_{f})_y & (H_{f})_x
			\end{pmatrix}\begin{pmatrix}f_{xx} & f_{xy} \\f_{yx} & f_{yy}\end{pmatrix}\begin{pmatrix}
			-(H_{f})_y\\ \; \, (H_{f})_x
		\end{pmatrix}$; and 
			\item [ii)] $\begin{pmatrix}
			E_{1,f}\\ E_{2,f}
			\end{pmatrix}=\begin{pmatrix}f_{xx} & f_{xy} \\f_{yx} & f_{yy}\end{pmatrix}\begin{pmatrix}
			-(H_{f})_y\\ \; \; (H_{f})_x
			\end{pmatrix}$,
		\end{itemize}\vspace{-0.4cm} 
		where $(H_{f})_x$ and $(H_{f})_y$ are the partial derivatives of the Hessian of $f$ with respect to $x$ and $y$, respectively.
\end{defi}

Note that
\begin{equation}\label{Rel.CfyQf}
C_f(x,y)=Q_f(-(H_{f})_y, (H_{f})_x),
\end{equation} 
where $Q_f$ is the quadratic form \vspace{-0.2cm} 
\begin{equation*}
	Q_f(x,y)=\begin{pmatrix}
dx&dy
\end{pmatrix}\begin{pmatrix}f_{xx} & f_{xy} \\f_{yx} & f_{yy}\end{pmatrix}\begin{pmatrix}
dx\\ dy
\end{pmatrix},
\end{equation*} while the polynomials $E_{1,f}$ and $E_{2,f}$ were introduced by V. I. Arnold in \cite{ArnoldRemarksParabCurves}.

\begin{prop}\label{PPE=CfyHf}
	Let $S\subset \R^3$ be locally expressed as the \graf $\Gamma_f$
	of a \diff \fn   
	
	$f:\R^2\rightarrow \R$ and let $H_f$ be the Hessian of $f$.
	\begin{enumerate}\setlength\itemsep{0em}
	\item[i)] The projection on the $xy$-plane of the parabolic curve of $\Gdef$ is the Hessian curve of $f$. A tangent vector to the parabolic curve of $\Gdef$ at a point $p$ projects to a vector which is a multiple of the vector $(-(H_{f})_y(\pi(p)), (H_{f})_x(\pi(p))) \in \R^2$, tangent to the Hessian curve of $f$ at $q$.
	\item[ii)] Let $l : \R \rightarrow \R^3$, $t\mapsto p+tu$ with $u\in \R^3$, parametrise a line with contact order $k \geq 2$ at $p \in \Gamma_f$. Then $l$ is an asymptotic direction of $\Gamma_f$ if and only if the projection $\pi(u)$ is a zero of the quadratic form $Q_f$.
	\item[iii)] If the Hessian curve of $f$ is non-singular, then the set of special parabolic points in $\Gdef$ is defined by the intersection of the tangent curves $V(H_f)$ and $V(C_{f})$. 
	\end{enumerate}
\end{prop} 

The proof of i) and ii) are straight forward and part iii) is given in \cite{OrtizHdezAffinegeom}.

Our next result allows us to find special parabolic points in the \graf\,of a \diff function in terms of its Hessian curve and the curves $V(E_{1,f})$ and $V(E_{2,f})$.

\begin{prop}\label{PPE=H,E1,E2=0}
	Let $p=(q,f(q))\in \R^3$ be in the \graf $\Gdef$ 
	of a \diff \fn $f:\R^2\rightarrow \R$. If the Hessian curve of $f$ is non-singular at $q$, then $p\in \Gdef$ is an special parabolic point \ssi $q$ lies in the intersection of the curves $V(H_f)$, $V(E_{1,f})$ and $V(E_{2,f})$.
\end{prop}

	\vspace{-0.6cm}
	
\begin{proof}
	We will prove the forward implication. From Proposition \ref{PPE=CfyHf}, $p=(q,f(q))\in \Gdef$ is a special parabolic point \ssi $q\in V(H_{f})\cap V(C_{f})$. The condition $Q_f(-(H_{f})_y(q), (H_{f})_x(q)) = C_{f}(q) = 0$ given by \eqref{Rel.CfyQf} implies, following Prop. \ref{PPE=CfyHf} \,\textrm{ii)}, that the line $l$ in the tangent plane to $\Gdef$ passing through $p\in \Gdef$ in the direction $u\in\R^3$, with $(0, 0)\neq \pi(u)=(-(H_{f})_y(q), (H_{f})_x(q))\in \R^2$, is an asymptotic direction of $\Gdef$ at $p$.
	
	The vector $\pi(u)=(-(H_{f})_y(q), (H_{f})_x(q))$ is the only zero in $\R^2\setminus \{(0,0)\}$ of the normal curvature function $v\mapsto C_f(v)$. Since $\R^2\setminus \{(0,0)\}$ is homotopically equivalent to $S^1$, then 
	 $\pi(u)$ is either a maximum or a minimum of the normal curvature function and, thus, $\pi(u)$ is the only eigenvector of the Hessian matrix of $f$. Let $\lambda$ be the eigenvalue associated to $\pi(u)$, then we have 
	$\displaystyle
	0=\pi(u)\begin{pmatrix}f_{xx} & f_{xy} \\f_{yx} & f_{yy}\end{pmatrix} \pi(u)^{\mathrm{t}}=\lambda \parallel \pi(u)\parallel^2$.
	Since $\pi(u)\neq (0,0)$, this implies that \vspace{-0.2cm} \begin{equation*}
	\begin{pmatrix}
	E_{1,f}(q)\\ E_{2,f}(q)
	\end{pmatrix}=\begin{pmatrix}f_{xx} & f_{xy} \\f_{yx} & f_{yy}\end{pmatrix}\begin{pmatrix}
	-(H_{f})_y(q)\\ \; \; (H_{f})_x(q)
	\end{pmatrix}=\begin{pmatrix}
	0\\ 0
	\end{pmatrix}.
	\end{equation*} 

 The backward implication is straightforward from Definition \ref{Ei's}.
 	
	\vspace{-0.6cm}
		
\end{proof}

\begin{cor}{\rm
Let $f$ be a \diff \fn $f:\R^2\rightarrow \R$ and let $\Gdef$ be the graph of $f$. If the Hessian curve of $f$ is non-singular, then 
	\vspace{-0.4cm}
	
	\begin{equation*}
	SPP(f)=\{p\in \Gdef; \pi(p)\in V(H_{f})\cap V(E_{1,f})\cap V(E_{2,f})\}.
	\end{equation*} 
}
\end{cor}

\begin{proof}
It follows from Proposition \ref{PPE=H,E1,E2=0}.
\end{proof}

\section{Viro's patchworking}\label{sec:Viro'spatchworking}

In this section we will recall Viro's patchworking technique. This procedure was introduced in the late seventies as a technique to glue simple algebraic curves in order to construct real algebraic non-singular curves with prescribed topology. 
	
Let $ f\in \R[x,y]$ be a polynomial, the \textit{support} of $f$ is the finite set of pairs $(i,j)\in \Z^{2}$ whose entries are the exponent of a monomial in $f$. That is, given $\displaystyle f(x,y):=\sum a_{i,j}x^{i}y^{j}$, 
\begin{equation*}
	\Supp (f):=\left\{(i,j)\in \Z^2 ; a_{i,j}\neq 0 \right\}.
\end{equation*}
	
For any subset $A\subset \R^{2}$, we define the \textit{restriction of $f$ to $A$} by 
\begin{equation*}
	\displaystyle f|_A (x,y):=\sum_{(i,j)\in A\cap \Supp (f)} a_{i,j}x^{i}y^{j}.
\end{equation*}


Let $\Delta \subset \R^2$ be a polyhedron and let $\tau$ be a polyhedral subdivision of $\Delta$. We say that $\tau$ is \textit{convex} if there exists a convex piecewise linear function $\lambda: \Delta\to \R_{\geq 0}$, taking integer values on the vertices of the subdivision $\tau$, whose restriction to the polyhedra of $\tau$ is linear; and with the property that it is not linear in the union of any two distinct polyhedra of $\tau$. We will say in this case that $\lambda$ \emph{induces the convex polyhedral subdivision} $\tau$.

	Given a convex polyhedral subdivision $\tau$ induced by the function $\lambda$, the \graf $\Gamma_{\lambda}$ forms a polytope called the \textit{compact polytope with polyhedral subdivision induced by $\lambda$}. We will refer to the set of 2-dimensional faces that lie in $\Gamma_{\lambda}$ by $T(\lambda)$.

The projection $\pi:\R^3 \to \R^2$ on the $xy$-plane induces a bijection between the faces of $T(\lambda)$ and the polyhedra in $\tau$. The inverse of this bijection will be denoted by $\mu$, that is,
\begin{equation*}
	\mu : \tau \to T(\lambda), \, E \mapsto \pi^{-1}(E)\cap \Gamma_{\lambda}.
\end{equation*}	

Let $\Delta \subset \R^2$ be a polyhedron and let $\lambda :\Delta \to \R_{\geq 0}$ be a convex linear function inducing $\tau$, a convex polyhedral subdivision of $\Delta$. Let $ \displaystyle f(x,y)=\sum_{(i,j)\in \Delta} a_{i,j}x^{i}y^{j}\in \R[x,y]$ be a polynomial whose support is contained in $\Delta$. The polynomial $\displaystyle f_{t}(x,y):=\sum_{(i,j)\in \Delta} a_{i,j}t^{\lambda(i,j)}x^{i}y^{j}\in\R[t][x,y]$, will be called the \textit{patchworking polynomial of $f$ induced by $\lambda$}. Given $\widetilde{S}\subset \R^{3}$, the restriction of $f_{t}$ to $\widetilde{S}$ is given by
\begin{equation*}
	f_{t}|_{\widetilde{S}}:=\sum_{\left(i,j,\lambda(i,j)\right)\in \widetilde{S}}a_{ij}t^{\lambda(i,j)}x^{i}y^{j}.
\end{equation*}

Given $\widetilde{E}\in T(\lambda)$ and $E\in \tau$ such that $\pi(\widetilde{E})=E$, then $f_t |_{\widetilde{E}}$ is the patchworking of $f_{E}$ induced by $\lambda|_{E}$.
			
\begin{defi}
	Let $\Delta \subset \R^2$ be a polyhedron and let $\tau$ be a convex polyhedral subdivision of $\Delta$ induced by $\lambda :\Delta \to \R_{\geq 0}$. Let $ \displaystyle f(x,y)=\sum_{(i,j)\in \Delta} a_{i,j}x^{i}y^{j}\in \R[x,y]$ be a polynomial with support contained in $\Delta$ and let $f_t$ be the patchworking polynomial of $f$ induced by $\lambda$. 
	Given $r\in \R_{\geq 0}$, we will denote by $f_t^{[r]}$ the restriction of $f_t$ to $\lambda^{-1}(r)$, i.e.,
	\begin{equation*}
		f^{[r]}_t(x,y):=t^r\!\!\!\!\!\! \!\! \!\! \!\!\!\!\sum_{\{(i,j)\in \Delta;\lambda(i,j)=r\}}\!\! \!\! \!\! \!\!\!\! \!\!a_{i,j}x^{i}y^{j}.
	\end{equation*}
\end{defi}			

Set $\displaystyle r:=\!\!\!\!\min_{(i,j)\in \Supp (f)}\!\!\!\!\!\!\!\lambda (i,j)$, then $\widetilde{E}:=\left\{ (a,b,c)\in \Gamma_{\lambda} ; c=r\right\}$ is a face of $T(\lambda)$ and 

\begin{equation}
	f_{t}|_{\widetilde{E}}=f_{t}^{[r]}=t^{r}f|_{\pi(\widetilde{E})}=t^{r}f|_{E}, \mbox{ where } E=\pi(\widetilde{E}).
\end{equation}


	
	

Let $\Delta \subset \R^2$ be a polyhedron with vertices in the integer lattice $\Z^2$ and let $f \in\R[x,y]$ be a polynomial with support in $\Delta$. Let $\tau$ be the convex polyhedral subdivision of $\Delta$ induced by $\lambda :\Delta \to \R_{\geq 0}$. Denote by $CC(f)^{*}$ the set of compact connected components of $V(f) \setminus \{xy=0\}$ and define
\begin{equation*}
	CC(f,\tau)^{*}:=\bigcup_{E\in \tau}\{(E,\mathcal{C}); \mathcal{C}\in CC(f|_{E})^{*} \}.
\end{equation*}

Viro's construction implies that, under some generic conditions, if $f_t \in\R[t][x,y]$ is the patchworking polynomial of $f$ induced by $\lambda$, then there exists $\delta>0$ such that there is an inclusion \begin{equation*}
CC(f,\tau)^{*}\hookrightarrow CC(f_t)^{*}
\end{equation*}\vspace{-0.8cm} 

for $0<|t|<\delta$.	The main purpose of this article is to extend this result to special parabolic points. 

\begin{ex}{\rm Set $f(x,y):=x^2y^2(1+x+y+y^2)$, let $\Delta:=Conv\left(\{(2,2),(3,2),(2,3),(2,4)\}\right)$ be the Newton polyhedron associated to $f$, and 
let $\lambda :\Delta \to \R_{\geq 0}$ be the convex function defined as follows
\begin{equation*}
	\lambda(i,j)= \left\{ \begin{array}{lcc}
 0 & if & i+j \leq 5, \\
 \\ i+j-5 & if & i+j > 5. 
 \end{array}
 \right.
\end{equation*}

 The subdivision of $\Delta$ induced by $\lambda$, is $\tau:=\{\Delta_1, \Delta_2\}$ where $\Delta_1:=Conv\left(\{(2,2),(3,2),(2,3)\}\right)$ and $\Delta_2:=Conv\left(\{(3,2),(2,3),(2,4)\}\right)$. The patchworking polynomial of $f$ induced by $\lambda$ is $f_{t}=x^{2}y^{2}(1+x+y+ty^{2})$. For $0<|t|<0.3$, we have the following pictures.

\begin{figure}[H]
	\begin{center}
		 \includegraphics[height=4cm]{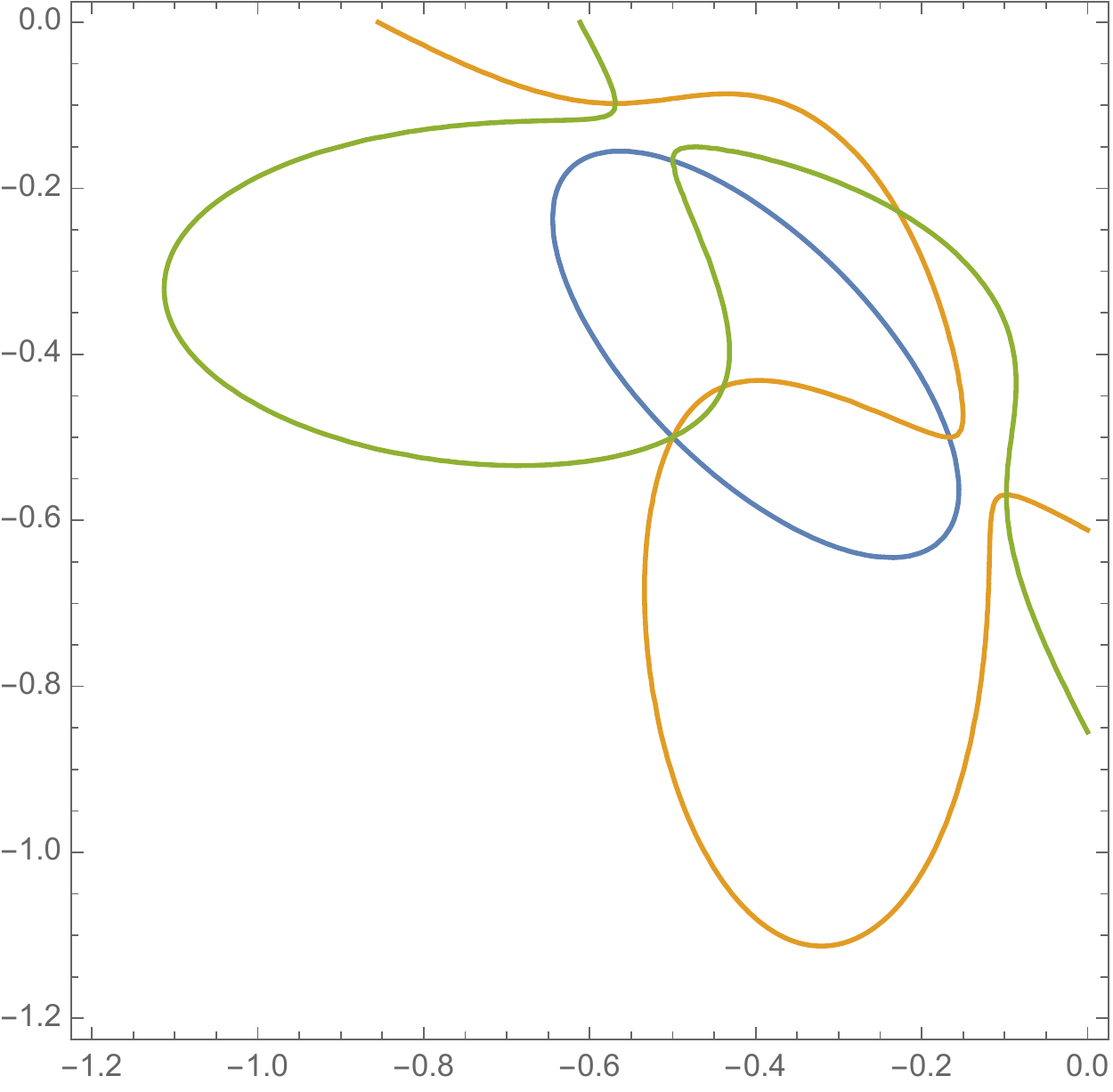}
		 \put(-100,-10){a)}
		 \hspace{1cm} 
		 \includegraphics[height=4cm]{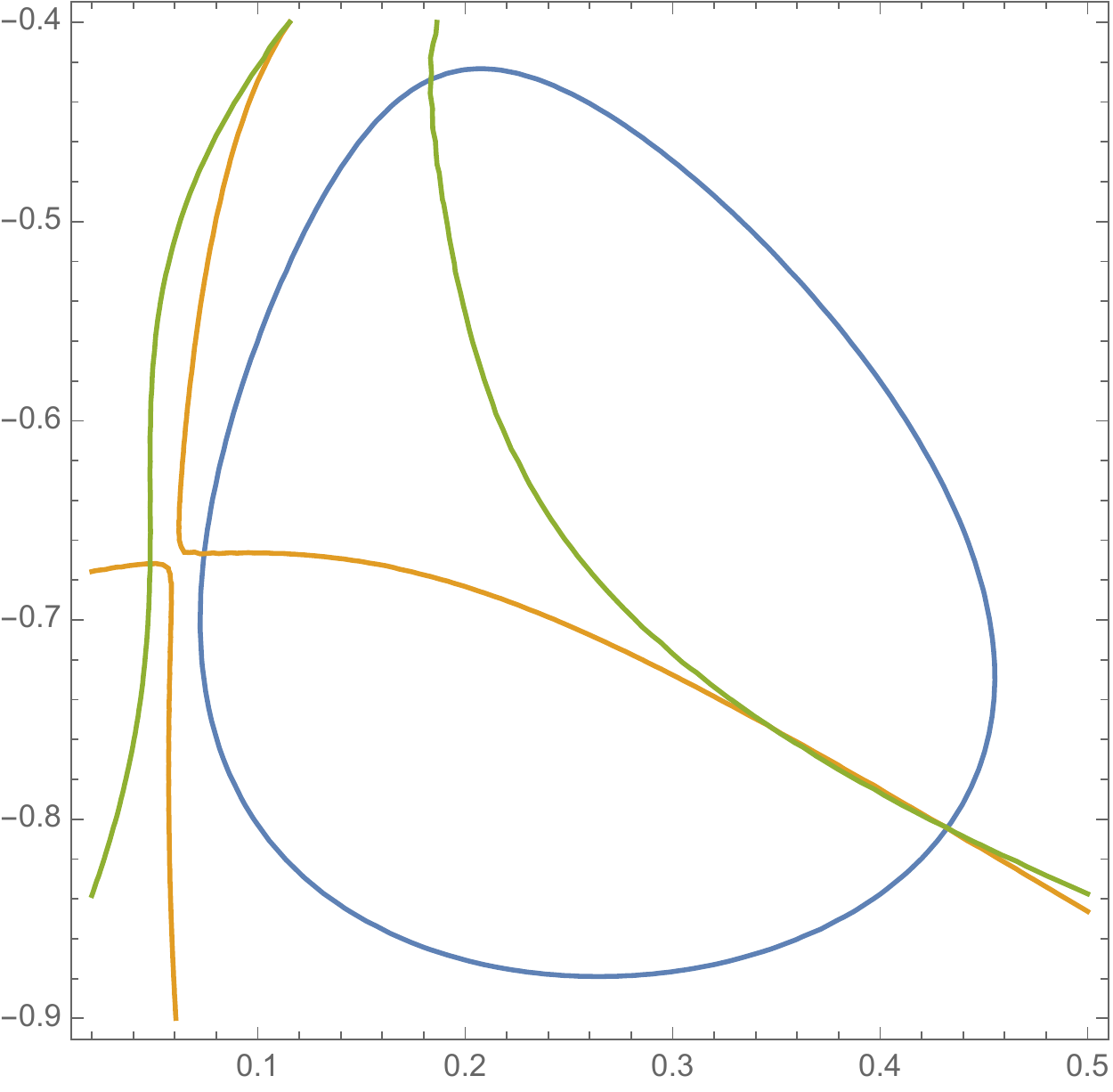}
		 \put(-100,-10){b)} 
		 \hspace{1cm} 
		 \includegraphics[height=4cm]{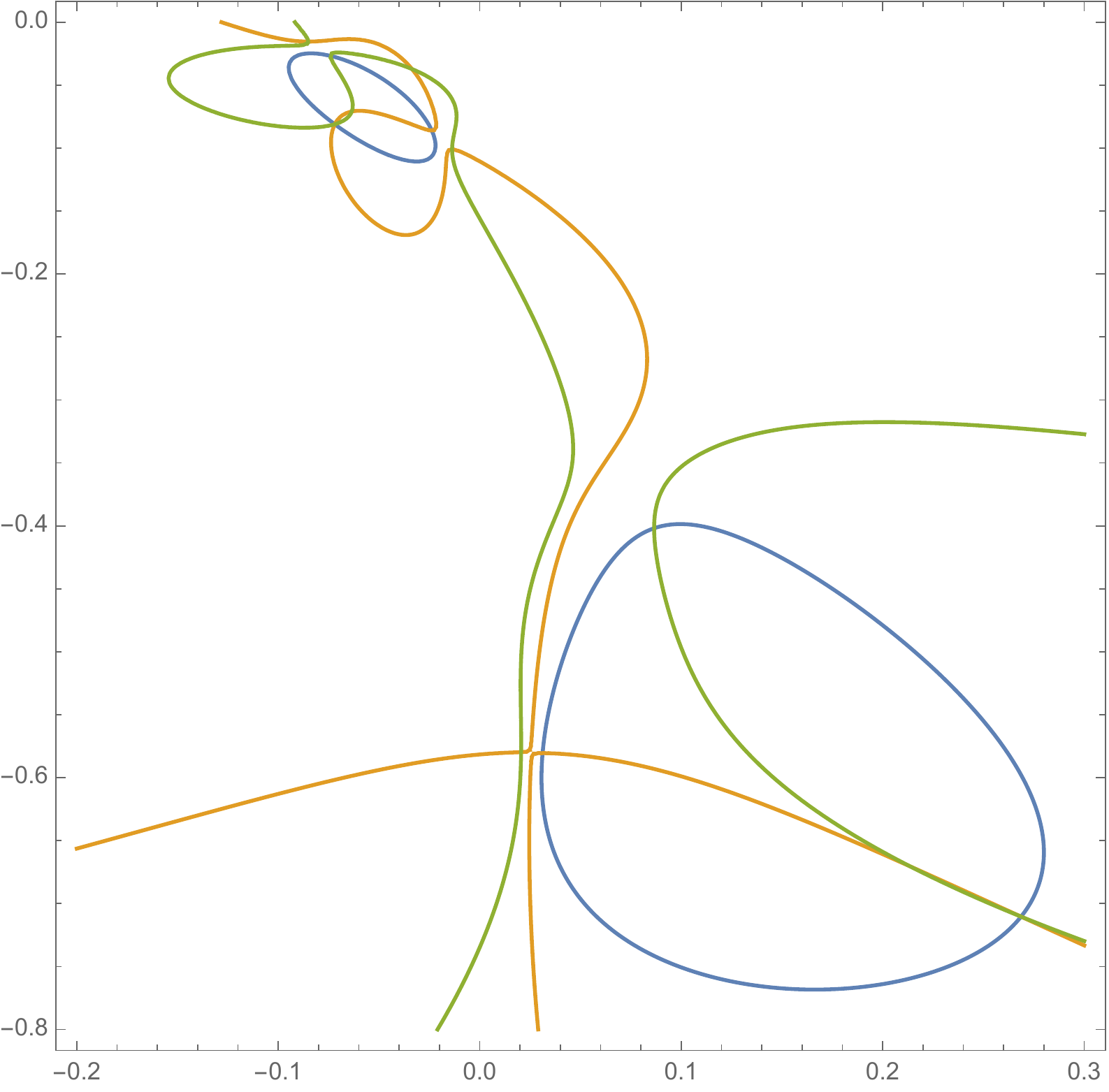}
		 \put(-100,-10){c)} 
		\label{pegado}
	\end{center}
	\caption{Figure a) shows $E1_{f|_{\Delta_1}}, E2_{f|_{\Delta_1}}$ and $H_{f|_{\Delta_1}}$; figure b) shows $E1_{f|_{\Delta_2}}, E2_{f|_{\Delta_2}}$ and $H_{f|_{\Delta_2}}$; and figure c) shows $E1_{{f_t}|_{\Delta}}, E2_{{f_t}|_{\Delta}}$ and $H_{{f_t}|_{\Delta}}$. 
	}
\end{figure}

}
\end{ex}



\section{On Perturbation Theory}\label{sec:Deftheory}

In this section we recall some definitions and statements on perturbation theory of polynomials and curves. These statements are consequences of general results in differential topology (see for example \cite{Hirsch1994DiffTop}) and are closely related to the concept of transversality. Detailed proofs are also written in \cite{Angelito}.

Two non-empty curves $C_1, C_2 \subset \R^2$ \textit{intersect transversally at $q\in C_1\cap C_2$}, denoted by $C_1 \pitchfork_{\, q} C_2$, if they are non-singular at $q$ and their tangent lines at $q$ are transversal. 

\begin{defi}{\rm Let $p=(q,f(q))\in \R^3$ be a special parabolic point of the graph of $f\in \R[x,y]$. We will say that $p$ is a \textit{transversal special parabolic point} of $f$ if the curve $V(H_f)$ is non-singular at $q\in\R^2$ and $V(E_{1,f})\pitchfork_{\, q} V(E_{2,f})$. We will denote by $TSPP(f)$ the set of transversal special parabolic points of $\Gamma_{f}$.
	 }
\end{defi}

For $i=1,2$ the tangent space to any smooth point $(x,y) \in \R^2$ in the curve $V(E_{i,f})$ is given by $\ker dE_{i,f}|_{(x,y)}$. A smooth point $p=(q,f(q))\in \R^3$ is a transversal \spp in the \graf of $f$, if \vspace{-0.7cm}

\begin{equation}\label{TangSpace}
\R^2 \simeq\ker dE_{1,f}|_{q}+\ker dE_{2,f}|_{q}. 
\end{equation}

Platonova's genericity condition \cite{PlatonovaSingularities} implies that special parabolic points are generically transversal.

Let $f\in \R[x,y]$ be a polynomial. A family of functions $F_t=f+tg_t$, where $g_t\in \R[t][x, y]$, will be called a \textit{perturbation} of $f$. Let $C :=V(f)$ be the curve defined by the zero set of $f\in \R[x,y]$, the family of curves $C_t := V(F_t)$, defined by a perturbation of $f$ will be called a \textit{perturbation of $C$}.	

Given a point  $q\in\R^{2}$, we will denote by $D(q,r)$, the closed disc of radious $r$ centered at $q$. 
\begin{prop}\label{Prop:LimdePerturbacionenptslimite}
	Let $F_t\in \R[t][x, y]$ be a perturbation of $f\in \R[x,y]$. For $\delta >0$, let $\displaystyle \{q_t\}_{t\in (-\delta,\delta)}$ be a collection of points in $D(q,r)\subset \R^{2}$ such that $\displaystyle \lim_{t\rightarrow 0}q_t=q$, then $\displaystyle \lim_{t\rightarrow 0}F_t(q_t)=f(q)$.		
\end{prop}

\begin{prop}\label{Prop:transpreservedunderperturb}
	Let $C_t, D_t\subset \R^2$ be perturbations of the curves $C, D \subset \R^2$. If $C \pitchfork_{\, q} D$, then for any $r>0$ there exists $\delta>0$ such that, for $|t|<\delta$, the curves $C_t$ and $D_t$ intersect transversally at some point $q_t\in D(q,r)$. Moreover, $q_t$ can be chosen so that $\displaystyle \lim_{t\rightarrow 0}q_t=q$.
\end{prop}

The condition of transversality is crucial in Proposition \ref{Prop:transpreservedunderperturb}.
\begin{ex}\label{exampleTangintersection}{\rm
Let $C_t:=V(y-x^2-t), D_t:=V(y+x^2+t)$ be perturbations of the curves $C:=V(y-x^2), D:=V(y+x^2)$. The curves $C$ and $D$ have one intersection point, but $C \cap_{\, \underline{0}} D$, is non-transversal. For positive values of $t$ the intersection $C_{t} \cap_{\, \underline{0}} D_{t}$ inside $D(\underline{0},r)$ is empty, while for negative values of $t$ the intersection $C_{t} \cap_{\, \underline{0}} D_{t}$ inside $D(\underline{0},r)$ has two points, so the number of points in the intersection of $C$ and $D$ is not preserved under small perturbations.
}
\end{ex}

Proposition \ref{Prop:transpreservedunderperturb} cannot be extended to more than two curves.

\begin{ex}\label{exampleTrintersection}{\rm 
Let $C_t:=V(x-t), D_t:=V(y-t)$ and $E_t:=V(y+x+t)$ be perturbations of the curves $C:=V(x), D:=V(y)$ and $E:=V(y+x)$. The curves $C,D$ and $E$ have only one transversal intersection point. For small values of $|t|$ the intersection $C_{t}\cap D_{t}\cap E_{t}$ inside $D(\underline{0},r)$ is empty, so the number of points in the intersection of $C,D$ and $E$ is not preserved under small perturbations.
}
\end{ex}

\begin{prop}\label{Prop:NonsingPerturbationcurve}{\rm Let $C_t:=V(F_t)\subset \R^2$ be a perturbation of the curve $C:=V(f)\subset \R^2$. Let $C$ be non-singular inside $D(q,R)$ for $q\in C$. Then there exists $\delta>0$ such that for $|t|<\delta$, the intersection $C_t\cap D(q,R)$ is non-empty and non-singular. 
		
	}
\end{prop}

Given a non-empty subset $A\subset \R^2$ and $\varepsilon>0$, we will denote by $\Tub_{\varepsilon}(A)\subset \R^2$, the set of points whose distance to $A$ is no greater than $\varepsilon$ and call it the \textit{tubular neighbourhood} of radious $\varepsilon$ centered along $A$, that is,
\begin{equation*}
	 \Tub_{\varepsilon}(A):= \bigcup_{q\in A}D(q,\varepsilon).
\end{equation*}
We will denote by $\mathrm{Int} \; \Tub_{\varepsilon}(A)$, the interior of the tubular neighbourhood $\Tub_{\varepsilon}(A)$.


\begin{prop}\label{Prop:NonsingPerturbationcurve1}
	Let $C_t:=V(F_t)\subset \R^2$ be a perturbation of the curve $C:=V(f)\subset \R^2$. Given $\varepsilon>0$, $R>0$ and $q\in C$, there exists $\delta>0$ such that for $|t|<\delta$, the intersection $C_t\cap D(q,R)$ is contained in the tubular neighbourhood $\Tub_{\varepsilon}(C)$.
\end{prop}

		
			

\section{Transversal special parabolic points under perturbation of functions}\label{sec:SPPsunderperturbation}

Special parabolic points can be determined by the intersection of two tangent curves (see Proposition \ref{PPE=CfyHf} \,\textrm{iii)}), or the intersection of three curves (see Proposition \ref{PPE=H,E1,E2=0}); however, as we have seen in examples \ref{exampleTangintersection} and \ref{exampleTrintersection}, both of these situations are generally not preserved under small perturbations.

 In \cite{E.Landis}, E. Landis states that, under some general conditions, special parabolic points are preserved under perturbations. He doesn't give a proof of this fact. 
 We gather that this fact is a consequence of Platonova's work. However, here we give a detailed proof for transversal special parabolic points. 

\begin{prop}\label{Prop:H,E1,E2underPerturbations}
	Let $f\in \R[x,y]$ be a polynomial of degree $d\geq 3$. 
		If $V(F_{t})\subset \R^2$ is a perturbation of the curve $V(f)$, then the curves defined by the polynomials $H_{F_t}$, $E_{1,F_{t}}$ and $E_{2,F_{t}}$, 
		are perturbations of the curves defined by $H_f$, $E_{1,f}$ and $E_{2,f}$, respectively.
\end{prop}
\begin{proof}
	The Hessian of $F_{t}(x,y)=f(x,y)+t\,g_t(x,y)$ is given by $H_{F_{t}}(x,y)= H_{f}(x,y)+t\widetilde{h}_t(x,y)$, where $\widetilde{h}_t:=(f_{xx}{g_t}_{yy}+{g_t}_{xx}f_{yy}-2f_{xy}{g_t}_{xy})+tH_{g_t}$ and $H_{f}$, $H_{g_t}$ are the Hessians of $f$ and $g_t$, respectively. Hence, the curve $V(H_{F_t})$ is a perturbation of the curve $V(H_{f})$.
	
	By definition, the polynomials $E_{1,F_{t}}$ and $E_{2,F_{t}}$, are given by
	\begin{align*}
		E_{1,F_{t}}(x,y):&= (-H_{F_{t}})_y(F_{t})_{xx}+(H_{F_{t}})_{x}(F_{t})_{xy}=E_{1,f}(x,y)+t\widetilde{e_1}_t(x,y) \; \mbox{and}\\
		E_{2,F_{t}}(x,y):&= (-H_{F_{t}})_y(F_{t})_{xy}+(H_{F_{t}})_{x}(F_{t})_{yy}=E_{2,f}(x,y)+t\widetilde{e_2}_t(x,y),
	\end{align*}	
	where $\widetilde{e_1}_t(x,y)=\psi_{1}+t(\psi_{2}+tE_{1,g_t})$ and $\widetilde{e_2}_t(x,y)=\xi_{1}+t(\xi_{2}+tE_{2,g_t})$ are perturbations of $\psi_{1}=-(H_f)_y{g_t}_{xx}+(H_f)_x{g_t}_{xy}-\varphi_yf_{xx}+\varphi_{x}f_{xy}$ and $\xi_{1}=-(H_f)_y{g_t}_{xy}+(H_f)_x{g_t}_{yy}-\varphi_yf_{xy}+\varphi_{x}f_{yy}$, respectively. The remaining polynomials are given by $\psi_2=-\varphi_{y}{g_t}_{xx}+\varphi_{x}{g_t}_{xy}-(H_{g_t})_yf_{xx}+(H_{g})_xf_{xy}$ and $\xi_{2}=-\varphi_{y}{g_t}_{xy}+\varphi_{x}{g_t}_{yy}-(H_{g_t})_yf_{xy}+(H_{g_t})_xf_{yy}$, which are given in terms of the Hessians of $f$ and $g_t$. This way, the curves $V(E_{1,F_t})$ and $V(E_{2,F_t})$ are perturbations of the curves $V(E_{1, f})$ and $V(E_{2, f})$, respectively, as claimed.	
\end{proof}

\begin{cor}\label{HFt non-sing}
	Let $f\in \R[x,y]$ be a polynomial of degree $d\geq 3$ and let $\Omega \subset \R^2$ be a bounded region in $\R^2$. If $F_t$ is a perturbation of $f$ and the curve $V(H_f)$ is non-singular inside the closure of $\Omega$, then the Hessian curve $V(H_{F_t})$ of $F_t$ is non-singular in $\Omega$ for sufficiently small values of $t$.
\end{cor}
\begin{proof} Direct consequence of Propositions \ref{Prop:NonsingPerturbationcurve} and \ref{Prop:H,E1,E2underPerturbations}.
\end{proof}

\begin{lemma}\label{q_tenV^k_tNonsing,sik_tneq0}
	Let $q$ be a point in $V(E_{1,f})\cap V(E_{2,f})$ and let $k=H_{f}(q)$. If $k\neq 0$, then $q$ is a singular point of the level set curve $H^{-1}_{f}{(k)}=\{(x,y)\in\R^2; H_{f}(x,y)=k\}$ defined by the Hessian of $f$.
\end{lemma} 
\begin{proof}
	Take $q\in V(E_{1,{f}})\cap V(E_{2,{f}})$ and suppose that $q\in H^{-1}_{f}{(k)}$ with $k\neq 0$. Since \begin{equation*}
	\begin{pmatrix}
	0\\ 0
	\end{pmatrix}=\begin{pmatrix}
	E_{1,{f}}(q)\\ \, E_{2,{f}}(q)
	\end{pmatrix}=\begin{pmatrix}f_{xx}(q) & f_{xy}(q) \\f_{yx}(q) & f_{yy}(q)\end{pmatrix}\begin{pmatrix}
	-(H_{{f}})_y(q)\\ \; \; (H_{{f}})_x(q)
	\end{pmatrix},
	\end{equation*}
	the vector $\upsilon(q):=(-(H_{{f}})_y(q), (H_{{f}})_x(q))^{\mathrm{t}}\in \ker\begin{pmatrix}f_{xx}(q) & f_{xy}(q) \\f_{yx}(q) & f_{yy}(q)\end{pmatrix}\cap T_{q} H^{-1}_{f}{(k)}$, where $T_{q} H^{-1}_{f}{(k)}$ is the tangent space to $H^{-1}_{f}{(k)}$ at ${q}$. If $\upsilon(q)\neq (0,0)$, then 
	$\rank \begin{pmatrix}f_{xx}(q) & f_{xy}(q) \\f_{yx}(q) & f_{yy}(q)\end{pmatrix} \leq 1$ so $q\in V(H_{f})$ and we reach a contradiction. Hence $\upsilon(q)= (0,0)$, which implies that $q$ is a singular point of $H^{-1}_{f}{(k)}$.
\end{proof}	

Our next theorem allows us to relate the transversal special parabolic points in the graph of a function to those in the graph of any of its perturbations. 

\begin{thm}\label{Thm:TSPPdeFtclosetoTSPPdef}
	Let $f\in \R[x,y]$ be a polynomial of degree $d\geq 3$ with non-singular Hessian curve and let $F_t$ be a perturbation of $f$. 
		Then, for every $(q,f(q))\in TSPP(f)$ and $\forall \varepsilon>0$, there exists $\delta>0$ such that for $|t|<\delta$ there is a point $(q_t,F_t(q_t))\!\in TSPP(F_t)$ with $q_t$ inside the closed disc $D(q,\varepsilon)\!\subset \R^2$ of radious $\varepsilon$.
\end{thm}
\begin{proof}
	Let $p=(q,f(q))\in \R^3$ be a transversal \spp in the \graf of $f$ and let $\pi:\R^3\rightarrow \R^2$ be the projection on the $xy$-plane. 
	
	Since $p$ is a transversal special parabolic point, then $V(H_f)$ is non-singular at $q$. By Corollary \ref{HFt non-sing}, the Hessian curve of $F_t$ is non-singular inside $D(q,\varepsilon)$ for small values of $t$, and by Proposition \ref{PPE=H,E1,E2=0} there exists $\delta_1>0$ so that, for $|t|<\delta_1$, 
	\begin{equation*}
		\pi\left(TSPP(F_t)\right)\cap D(q,\varepsilon)=V(H_{F_t})\cap V(E_{1,{F_t}})\cap V(E_{2,{F_t}})\cap D(q,\varepsilon).
	\end{equation*}

	By Proposition \ref{Prop:transpreservedunderperturb}, there exists $\delta_2>0$ so that for $|t|<\delta_2$ the curves $V(E_{1,F_t})$ and $V(E_{2,F_t})$ intersect transversally at some $q_t\in D(q,\varepsilon)$ with $\displaystyle \lim_{t\rightarrow 0}q_t=q$. 
	
	We claim that $(q_t,F_t(q_t))\in TSPP(F_t)$. To prove our claim it is enough to show that $q_t\in V(H_{F_t})$. Suppose that $H_{F_t}(q)=k_t\neq 0$, by Lemma \ref{q_tenV^k_tNonsing,sik_tneq0}, $q_t\in D(q,\varepsilon)$ is a singular point of $V^{k_t}:=\{(x,y)\in\R^2; H_{F_t}(x,y)=k_t\}$. The vector $\upsilon(q_t):=(-(H_{{F_t}})_y(q_t), (H_{{F_t}})_x(q_t))=(0,0)$, defines then a sequence with
	\begin{equation*}
	(0,0)\neq (-(H_{f})_y(q), (H_{f})_x(q))=\lim_{t\rightarrow 0}(-(H_{f})_y(q_t), (H_{f})_x(q_t))\stackrel{Prop.\; \ref{Prop:LimdePerturbacionenptslimite}}{=}\lim_{t\rightarrow 0}\upsilon(q_t)=(0,0).
	\end{equation*} 
	Therefore for $|t|< min\{\delta_1,\delta_2\}$, the point $q_t\in V(E_{1,{F_t}})\cap V(E_{2,{F_t}})$ lies also in $V(H_{F_t})\cap D(q,\varepsilon)$ and $(q_t,F_t(q_t))\in SPP(F_t)$ is, henceforth, a transversal \spp in the \graf of $F_t$ with $q_t\in D(q,\varepsilon)$.	
\end{proof}

\begin{cor}\label{Cor:InclusionTSPPdefenTSPPdeFt}
	Let $f\in \R[x,y]$ be a polynomial of degree $d\geq 3$ with non-singular Hessian curve and let $F_t\in \R[t][x, y]$ be a perturbation of $f$. Then, for $\varepsilon >0$ there exists $\delta>0$ such that for $0<|t|<\delta$ there exists an inclusion	
	\begin{equation}
		\psi_t: TSPP(f)\hookrightarrow TSPP(F_t)
	\end{equation}
	such that $\pi(\psi_t(p))\in D(\pi(p),\varepsilon)$. Furthermore, choosing $\varepsilon$ small enough, we also have
	\begin{equation}
		\psi_t: TSPP(f)^{*}\hookrightarrow TSPP(F_t)^{*}.
	\end{equation}
\end{cor}

\begin{proof}
	The set $TSPP(f)$ is finite. Let $0<\varepsilon_{1}< \varepsilon$ be such that for any $q, q^{\prime}\in \pi \left(TSPP(f)\right)$ with $q\neq q^{\prime}$, $D(q,\varepsilon_{1})\cap D(q^{\prime},\varepsilon_{1})=\emptyset$ and $D(q,\varepsilon_1), D(q^{\prime},\varepsilon_1)\subset (\R^*)^2$. By Theorem  \ref{Thm:TSPPdeFtclosetoTSPPdef}, for any $q\in \pi(TSPP(f))$, there exists $\delta_q>0$ such that, for $|t|<\delta_q$, there is a point $\left(q_t, F_t( q_t)\right)\in TSPP(F_t)$ with $q_t\in D(q,\varepsilon_1)$. 
	
	Let $\displaystyle \delta:=\!\!\!\min_{q\in \pi \left(TSPP(f)\right)}\!\!\!\!\!\delta_q>0$ and for $p=(q,f(q))\in TSPP(f)$ define $\psi_t(p):=\left(q_t, F_t( q_t)\right)$. 
\end{proof}

\section{On quasihomothetic maps}\label{sec:Sppsofhomotheticpols0}

In this section, we show some properties of a special type of transformations, called \textit{quasihomotheties} by O. Y. Viro \cite{ViroIntrotoTopRealAlgVars}. Quasihomotheties are maps
\begin{align*}
\rho_{s}^{(\alpha,\beta)}:\R^2 & \to \R^2 \\\nonumber %
(x,y)& \mapsto (s^{\alpha}x,s^{\beta}y).
\end{align*}
for some $\alpha, \beta\in \Z$ and $s\in \R$. For $s\neq 0$ the function $\rho_{s}^{(\alpha,\beta)}$ is one-to-one, and the differential $d\rho_{s}^{(\alpha,\beta)}|_{(x,y)}$ corresponds to the isomorphism defined by the matrix 
	$
	\begin{pmatrix}
		s^{\alpha} &0\\
	0& s^{\beta}
	\end{pmatrix}$. 

\begin{lemma}\label{DifHomothetieskeeppointsaway} 
		Let $h_{s^{\alpha_i}} : \R \to \R$, $x\mapsto s^{\alpha_i}x$. 
		Given $\alpha_1 < \alpha_2$, for any two intervals $[a,b],[c,d]\subset\R\setminus \{0\}$,	
		there exists $\delta>0$ such that for $0<|s|<\delta$ we have 
		$h_{s^{\alpha_1}}([a,b])\cap h_{s^{\alpha_2}}([c,d])= \emptyset$.
\end{lemma}
\begin{proof}
	We will give the proof in the case where $[a,b],[c,d]\subset\R_{>0}$.
	Consider $\delta_{1}=(\frac{a}{d})^{\frac{1}{\alpha_2-\alpha_1}}$. 
	If $s>0$ or $\alpha_1, \alpha_2$ are even, then $h_{s^{\alpha_1}}([a,b])=[s^{\alpha_1}a,s^{\alpha_1}b]$ and $h_{s^{\alpha_2}}([c,d])= [s^{\alpha_2}c,s^{\alpha_2}d]$, thus choosing $0<|s|<\delta_{1}$ gives $s^{\alpha_2-\alpha_1}<\frac{a}{d}$, so $s^{\alpha_2}d<s^{\alpha_1}a$ and $h_{s^{\alpha_1}}([a,b])\cap h_{s^{\alpha_2}}([c,d])= \emptyset$. If $s<0$ and $\alpha_1+\alpha_2$ is odd, then either $h_{s^{\alpha_1}}([a,b])\subset \R_{>0}$ and $h_{s^{\alpha_2}}([c,d])\subset \R_{<0}$; or $h_{s^{\alpha_1}}([a,b])\subset \R_{<0}$ and $h_{s^{\alpha_2}}([c,d])\subset \R_{>0}$, so the result follows. If $s<0$ and $\alpha_1, \alpha_2$ are odd, then $h_{s^{\alpha_1}}([a,b])=[s^{\alpha_1}b,s^{\alpha_1}a]$ and $h_{s^{\alpha_2}}([c,d])= [s^{\alpha_2}d,s^{\alpha_2}c]$, hence $|s|<\delta_{1}$ gives $s>-(\frac{a}{d})^{\frac{1}{\alpha_2-\alpha_1}}$, thus $0<-s<(\frac{a}{d})^{\frac{1}{\alpha_2-\alpha_1}}$, 
	and $s^{\alpha_1}a<s^{\alpha_2}d$, so $h_{s^{\alpha_1}}([a,b])\cap h_{s^{\alpha_2}}([c,d])= \emptyset$. 

	
	For the cases where $[a,b]\subset\R_{>0}$ and $[c,d]\subset\R_{<0}$; $[a,b]\subset\R_{<0}$ and $[c,d]\subset\R_{>0}$; or $[a,b],[c,d]\subset\R_{<0}$, it is enough to consider $0<|s|$ smaller than $\delta_{\rm{2}}=|\frac{a}{c}|^{\frac{1}{\alpha_2-\alpha_1}}$, $\delta_{\rm{3}}=|\frac{b}{d}|^{\frac{1}{\alpha_2-\alpha_1}}$, $\delta_{\rm{4}}=(\frac{b}{c})^{\frac{1}{\alpha_2-\alpha_1}}$ to have $h_{s^{\alpha_1}}([a,b])\cap h_{s^{\alpha_2}}([c,d])= \emptyset$, respectively. Taking $\delta:=\min\{\delta_{\rm{1}}, \delta_{\rm{2}}, \delta_{\rm{3}}, \delta_{\rm{4}}\}$, the result follows.
\end{proof}

\begin{prop}\label{Prop:phitkeepspointsawayV2}
	Given $(\alpha,\beta)\neq (\alpha^{\prime},\beta^{\prime})\in \Z^{2}$, let $A, B\subset (\R^{*})^2$ be finite sets of points and let $\varepsilon>0$ be such that for all $q\in A\cup B$, the closed disc $D(q,\varepsilon)\subset (\R^*)^2$. Then, there exists $\delta >0$ such that for $0<|s|<\delta$, 
	\begin{equation*}
		 \displaystyle \rho_{s} ^{(\alpha,\beta)}\left(Tub_{\varepsilon}(A)\right)\bigcap \rho_{s} ^{(\alpha^{\prime},\beta^{\prime})}\left(Tub_{\varepsilon}(B)\right) =\emptyset.
	\end{equation*}
\end{prop}

\begin{proof}
Suppose that $\alpha \neq \alpha^{\prime}$ and let $\pi_1: \R^2 \to \R$ be the projection $(x,y)\mapsto x$. Let $[a,b],[e,f]\subset \R_{<0}$, and $[c,d],[g,h]\subset \R_{>0}$ be intervals such that 
\begin{equation*}
	\pi_{1}(Tub_{\varepsilon}(A))\subseteq [a,b]\cup [c,d] \mbox{ and } 
	\pi_{1}(Tub_{\varepsilon}(B))\subseteq [e,f]\cup [g,h].
\end{equation*} 

Then, \vspace{-0.4cm}
\begin{equation*}
\pi_{1}\left( \rho_{s}^{(\alpha,\beta)}(Tub_{\varepsilon}(A))\right)=\left\lbrace s^{\alpha}x; x\in \pi_{1}\left(Tub_{\varepsilon}(A)\right)\right\rbrace \subseteq h_{s^{\alpha}}([a,b])\cup h_{s^{\alpha}}([c,d]) \;\mbox{and}
\end{equation*}
\begin{equation*}
\pi_{1}\left(\rho_{s}^{(\alpha^{\prime},\beta^{\prime})}(Tub_{\varepsilon}(B)) \right)=\left\lbrace s^{\alpha^{\prime}}x; x\in \pi_{1}(Tub_{\varepsilon}(B))\right\rbrace \subseteq h_{s^{\alpha}}([e,f])\cup h_{s^{\alpha}}([g,h]).
\end{equation*}
Taking these intervals two-to-two, the result follows from Lemma \ref{DifHomothetieskeeppointsaway}. 
 
If, on the other hand, $\alpha = \alpha^{\prime}$ and $\beta \neq \beta^{\prime}$, then with the projection $\pi_2: \R^2 \to \R, (x,y)\mapsto y$ and a similar process we obtain the result wanted.
\end{proof}


\section{Special parabolic points under quasihomothetic maps}\label{sec:Sppsofhomotheticpols}

In this section we show how the number of transversal special parabolic points is preserved under quasihomotheties. 

Given $\alpha,\beta, r\in \Z$, we will consider the transformation
\begin{align}\label{homotenx,y}
\widetilde{h}_{(\alpha,\beta,r)} :\R[x,y] & \to \R[s][x,y]\\\nonumber
f(x,y)& \mapsto \widetilde{h}_{(\alpha,\beta,r)}(f)(x,y)=s^r f \circ\rho_{s} ^{(\alpha,\beta)}(x,y).
\end{align} Note that $\widetilde{h}_{(\alpha,\beta,r)}(f)$ does not define a perturbation of $f\in \R[x,y]$. However, if we consider the translation $T_1\!:\R[s][x,y] \to \R[s][x,y]$, $f(x,y,s)\mapsto f(x,y,s+1)$, then the composition $T_1\circ \widetilde{h}_{(\alpha,\beta,r)}(f)$ is a perturbation of the polynomial $f$.	

\begin{lemma} \label{Lem:SPPsigualesbajosr}
	Let $\Omega \subset \R^{2}$ be a bounded region. If the Hessian curve of $f\in \R[x,y]$ is non-singular inside the closure of $\Omega$, then for $s\neq 0$ and $r\in \Z$,
	 \begin{enumerate}{\footnotesize
		\item[i)]$(x,y,f(x,y))\in SPP(f)\cap \pi^{-1}(\Omega)$ \ssi $(x,y,\widetilde{h}_{(0,0,r)}(f)(x,y))\in SPP(\widetilde{h}_{(0,0,r)}(f))\cap \pi^{-1}(\Omega)$.
 		\item[ii)] $(x,y,f(x,y))\in SPP(f)^*\cap \pi^{-1}(\Omega)$ \ssi $(x,y,\widetilde{h}_{(0,0,r)}(f)(x,y))\in SPP(\widetilde{h}_{(0,0,r)}(f))^*\cap \pi^{-1}(\Omega)$.}
	\end{enumerate}
\end{lemma}

\begin{proof} 
	Set $\widetilde{f}(x,y):=\widetilde{h}_{(0,0,r)}(f)(x,y)=s^rf(x,y)$, and let $\pi:\R^3\rightarrow \R^2$ be the projection on the $xy$-plane, by Proposition \ref{PPE=H,E1,E2=0} the set of special parabolic points in the \graf of $\widetilde{f}$ is given by the set 
	$SPP(\widetilde{f})\cap \pi^{-1}(\Omega)=\{(x,y,\widetilde{f}(x,y));(x,y)\in V(H_{\widetilde{f}})\cap V(E_{1,\widetilde{f}})\cap V(E_{2,\widetilde{f}})\}\cap \pi^{-1}(\Omega)$, 
	where \vspace{-0.2cm}
	\begin{align}\label{e1e2h0} 
		H_{\widetilde{f}}(x,y)\, =&\det \!\begin{pmatrix}\widetilde{f}_{xx}(x,y) & \widetilde{f}_{xy}(x,y) \\\widetilde{f}_{yx}(x,y) & \widetilde{f}_{yy}(x,y)\end{pmatrix}\!\!=\!\det \!\begin{pmatrix}s^r {f}_{xx}(x,y) & s^r {f}_{xy}(x,y) \\s^r {f}_{yx}(x,y) & s^r {f}_{yy}(x,y)\end{pmatrix}\!=\!s^{2r} H_{f}(x,y), \,\mbox{and}\\\nonumber\hspace{-0.5cm}
 		\begin{pmatrix}
 		\!E_{1,{\widetilde{f}}}(x,y)\\ \! E_{2,{\widetilde{f}}}(x,y)
 		\end{pmatrix} \!=&\begin{pmatrix}\widetilde{f}_{xx}(x,y) & \widetilde{f}_{xy}(x,y) \\\widetilde{f}_{yx}(x,y) & \widetilde{f}_{yy}(x,y)\end{pmatrix}\begin{pmatrix}
 		-(H_{{\widetilde{f}}})_y(x,y)\\ \; \; (H_{{\widetilde{f}}})_x(x,y)
 		\end{pmatrix}\\\label{e1e2h}
 		=&\begin{pmatrix}s^r {f}_{xx}(x,y) & s^r {f}_{xy}(x,y) \\s^r {f}_{yx}(x,y) & s^r {f}_{yy}(x,y)\end{pmatrix}\begin{pmatrix}
 		-s^{2r}(H_{{f}})_y(x,y)\\ \; \; s^{2r}(H_{f})_x(x,y)
 		\end{pmatrix}=s^{3r}\begin{pmatrix}
 		E_{1,f}(x,y)\\ \, E_{2,f}(x,y)
 		\end{pmatrix}.
	\end{align}
	
	Since the Hessian curve of $f$ is non-singular inside $\Omega$, by equation (\ref{e1e2h0}) also the Hessian curve of $\widetilde{f}$ is non-singular inside $\Omega$. Since we have $SPP(f)\cap \pi^{-1}(\Omega)=\{(x,y,f(x,y));(x,y)\in V(H_{f})\cap V(E_{1,f})\cap V(E_{2,f})\}\cap \pi^{-1}(\Omega)$, equations (\ref{e1e2h0}) and (\ref{e1e2h}) give the results wanted for $s \neq 0$.
\end{proof}

\begin{lemma}\label{Lem:SPPsigualesbajosrt}
	Let $\Omega \subset \R^{2}$ be a bounded region. If the Hessian curve of $f\in \R[x,y]$ is non-singular inside the closure of $\Omega$, then for $s\neq 0$ and $r\in \Z$,
	\begin{enumerate}{\footnotesize
		\item[i)] $\!\!(x,y,f(x,y))\!\in\! TSPP(f)\cap \pi^{-1}(\Omega)$ \ssi $(x,y,\widetilde{h}_{(0,0,r)}(f)(x,y))\!\in\! TSPP(\widetilde{h}_{(0,0,r)}(f))\cap \pi^{-1}(\Omega)$.
 		\item[ii)] $\!\!(x,y,f(x,y))\!\in\! TSPP(f)^*\cap \pi^{-1}(\Omega)$ \ssi $(x,y,\widetilde{h}_{(0,0,r)}(f)(x,y))\!\in\! TSPP(\widetilde{h}_{(0,0,r)}(f))^*\cap \pi^{-1}(\Omega)$.}
	\end{enumerate}
\end{lemma}

\begin{proof} A smooth point $p=(q,f(q))\in \R^3$ is a transversal \spp in the \graf of $f$ if it satisfies equation (\ref{TangSpace}). By equations (\ref{e1e2h0}) and (\ref{e1e2h}) we have that $\R^2\simeq \ker dE_{1,\widetilde{f}}|_{q}+\ker dE_{2,\widetilde{f}}|_{q}$, and by Lemma \ref{Lem:SPPsigualesbajosr} we have the result.
\end{proof}

Set $\widetilde{\varphi}^{(\alpha,\beta,r)}_s$ to be the transformation given by
\begin{align}\label{eq:2}
\widetilde{\varphi}^{(\alpha,\beta,r)}_s: \Gamma_{f} &\to \Gamma_{\widetilde{h}_{(\alpha,\beta,r)}(f)}\\
\left(x,y,f(x,y)\right) &\mapsto \left(s^{-\alpha}x,s^{-\beta}y,\widetilde{h}_{(\alpha,\beta,r)}(f)\left(\rho_{s} ^{(-\alpha,-\beta)}(x,y)\right)\right)\!. \nonumber
\end{align} 

\begin{prop}\label{Prop:BiyecciondeSPPsbajohomotecia}
Let $\Omega \subset \R^{2}$ be a bounded region. If the Hessian curve of $f \in \R[x,y]$ is non-singular inside the closure of $\Omega$, then for $s\neq 0$, and $\alpha,\beta, r\in \Z$, the mapping $\widetilde{\varphi}^{(\alpha,\beta,r)}_s$ gives a one-to-one correspondence between $SPP(f)\cap \pi^{-1}(\Omega)$
 and $SPP(\widetilde{h}_{(\alpha,\beta,r)}(f))\cap \pi^{-1}(\rho_{s}^{(-\alpha,-\beta)}(\Omega))$. Moreover, $\widetilde{\varphi}^{(\alpha,\beta,r)}_s$ gives a one-to-one correspondence between $SPP(f)^*\cap \pi^{-1}(\Omega)$
 and the set $SPP(\widetilde{h}_{(\alpha,\beta,r)}(f))^*\cap \pi^{-1}(\rho_{s}^{(-\alpha,-\beta)}(\Omega))$.
\end{prop}

\begin{proof} 
	By Lemma \ref{Lem:SPPsigualesbajosr}, it is enough to prove that $\widetilde{\varphi}^{(\alpha,\beta,r)}_s$ gives a one-to-one correspondence for $r=0$. Set $\widehat{f}(x,y) := \widetilde{h}_{(\alpha,\beta,0)}(f)(x,y) = f(\rho_{s} ^{(\alpha,\beta)}(x,y))$. By Proposition \ref{PPE=H,E1,E2=0}, the set of special parabolic points in the \graf of $\widehat{f}$ is described by the set 
	\begin{equation*}
		SPP(\widehat{f})\cap \pi^{-1}(\Omega)=\{(x,y,\widehat{f}(x,y)); (x,y)\in V(H_{\widehat{f}})\cap V(E_{1,\widehat{f}})\cap V(E_{2,\widehat{f}})\}\cap \pi^{-1}(\Omega),
	\end{equation*}	where
	\begin{align}\nonumber
		H_{\widehat{f}}(x,y)\;& =\det \begin{pmatrix}\widehat{f}_{xx}(x,y) & \widehat{f}_{xy}(x,y) \\
		\widehat{f}_{yx}(x,y) & \widehat{f}_{yy}(x,y)\end{pmatrix} \!=
		\det \begin{pmatrix}s^{2\alpha} {f}_{xx}(\rho_{s} ^{(\alpha,\beta)}(x,y)) & s^{\alpha+\beta} {f}_{xy}(\rho_{s} ^{(\alpha,\beta)}(x,y)) \\
		s^{\alpha+\beta} {f}_{yx}(\rho_{s} ^{(\alpha,\beta)}(x,y)) & s^{2\beta} {f}_{yy}(\rho_{s} ^{(\alpha,\beta)}(x,y))\end{pmatrix}\\\label{eq:22} 
		& =s^{2(\alpha+\beta)} H_{{f}}(\rho_{s}^{(\alpha,\beta)}(x,y)),\; \mbox{and} \\
	\nonumber
		\begin{pmatrix}
			E_{1,{\widehat{f}}}(x,y)\\ \, E_{2,{\widehat{f}}}(x,y)
			\end{pmatrix} \!&=\begin{pmatrix}\widehat{f}_{xx}(x,y) & \widehat{f}_{xy}(x,y) \\\widehat{f}_{yx}(x,y) & \widehat{f}_{yy}(x,y)\end{pmatrix}\begin{pmatrix} -(H_{{\widehat{f}}})_y(x,y)\\ \; \; (H_{{\widehat{f}}})_x(x,y)
		\end{pmatrix}\\\nonumber
		&=\begin{pmatrix}s^{2\alpha} {f}_{xx}(\rho_{s} ^{(\alpha,\beta)}(x,y)) & s^{\alpha+\beta} 								{f}_{xy}(\rho_{s}^{(\alpha,\beta)}(x,y)) \\s^{\alpha+\beta} {f}_{yx}(\rho_{s} ^{(\alpha,\beta)}(x,y)) & s^{2\beta} {f}_{yy}(\rho_{s} ^{(\alpha,\beta)}(x,y))\end{pmatrix}\begin{pmatrix}
		-s^{2(\alpha+\beta)} (H_{f})_y(s^{\alpha}x,s^{\beta}y)\\ \; \; s^{2(\alpha+\beta)} (H_{f})_x(s^{\alpha}x,s^{\beta}y)	\end{pmatrix} \\\label{eq:23} 
		& =\begin{pmatrix}
			s^{4\alpha+3\beta} E_{1,f}(\rho_{s} ^{(\alpha,\beta)}(x,y))\\ \, s^{3\alpha+4\beta} E_{2,f}(\rho_{s}^{(\alpha,\beta)}(x,y))
			\end{pmatrix}.
	\end{align}

	By equation \eqref{eq:22}, the Hessian curve $V(H_{\widehat{f}})$ of $\widehat{f}$ is non-singular inside $\Omega$ for $s\neq 0$; and by equations \eqref{eq:22} and \eqref{eq:23}, $\widetilde{\varphi}^{(\alpha,\beta,r)}_s$ is one-to-one between the sets $SPP(f)\cap \pi^{-1}(\Omega)$ and $SPP(\widetilde{h}_{(\alpha,\beta,r)}(f))\cap \pi^{-1}(\rho_{s}^{(-\alpha,-\beta)}(\Omega))$. Moreover, since for $s\neq 0$, $s^{\alpha}x=0$ \ssi $x\!=\!0$, and $s^{\beta}y=0$ \ssi $y\!=\!0$, the latest claim is proved.
\end{proof}

\begin{prop} \label{Prop:BiyecciondeTSPPsbajohomotecia}
	Let $\Omega \subset \R^{2}$ be a bounded region. If the Hessian curve of $f \in \R[x,y]$ is non-singular inside $\Omega$, then for $s\neq 0$ and $\alpha,\beta, r\in \Z$, $\widetilde{\varphi}^{(\alpha,\beta,r)}_s$ gives a one-to-one correspondence between $TSPP(f)\cap \pi^{-1}(\Omega)$
	and $TSPP(\widetilde{h}_{(\alpha,\beta,r)}(f))\cap \pi^{-1}(\rho_{s}^{(-\alpha,-\beta)}(\Omega))$. Moreover, $\widetilde{\varphi}^{(\alpha,\beta,r)}_s$ is a one-to-one correspondence between $TSPP(f)^*\cap \pi^{-1}(\Omega)$ and $TSPP(\widetilde{h}_{(\alpha,\beta,r)}(f))^*\cap \pi^{-1}(\rho_{s}^{(-\alpha,-\beta)}(\Omega))$.
\end{prop}

\begin{proof}
	By Corollary \ref{Lem:SPPsigualesbajosrt}, it is enough to prove the proposition for $r = 0$. Set $\widehat{f}(x,y) := \widetilde{h}_{(\alpha,\beta,0)}(f)(x,y) =f\circ \rho_s^{(\alpha,\beta)}(x,y)$. By Proposition \ref{Prop:BiyecciondeSPPsbajohomotecia}, it is enough to show that if $\R^2 \simeq\ker dE_{1,f}|_{(x,y)}+\ker dE_{2,f}|_{(x,y)}$, then there exists $\delta>0$ such that for $0<|s|<\delta$ we have $\R^2\simeq \ker dE_{1,\widehat{f}}|_{\rho_{s} ^{(-\alpha,-\beta)}(x,y)}+\ker dE_{2,\widehat{f}}|_{\rho_{s} ^{(-\alpha,-\beta)}(x,y)}$.	
	Since \begin{align*}
		dE_{1,\widehat{f}}|_{(x,y)}&= s^{4\alpha+3\beta} dE_{1,f}|_{\rho_{s}^{(\alpha,\beta)}(x,y)}\cdot
		\begin{pmatrix}
			s^{\alpha} &0\\
			0& s^{\beta}
		\end{pmatrix}\!, \,\mbox{and}\\
			dE_{2,\widehat{f}}|_{(x,y)}&= s^{3\alpha+4\beta} dE_{2,f}|_{\rho_{s}^{(\alpha,\beta)}(x,y)}\cdot
		\begin{pmatrix}
		s^{\alpha} &0\\
		0& s^{\beta}
		\end{pmatrix}\!,	
	\end{align*} 
then $\ker dE_{i,\widehat{f}}|_{\rho_{s} ^{(-\alpha,-\beta)}(x,y)}$ is the image of $\ker dE_{i,f}|_{(x,y)}$ under the isomorphism defined by the matrix $d\rho_s^{(\alpha,\beta)}|_{(x,y)}=
	\begin{pmatrix}
		s^{\alpha} &0\\
	0& s^{\beta}
	\end{pmatrix}\!$, for $\,i=1,2$. 
	Therefore, if $\R^2 \simeq\ker dE_{1,f}|_{(x,y)}+\ker dE_{2,f}|_{(x,y)}$, then $ \ker dE_{1,\widehat{f}}|_{\rho_{s}^{(-\alpha,-\beta)}(x,y)}+\ker dE_{2,\widehat{f}}|_{\rho_{s} ^{(-\alpha,-\beta)}(x,y)}\simeq \R^2$.	
\end{proof}

Let $f\in \R[x,y]$ be a polynomial and let $f_t$ be a perturbation of $f$. For $t\in \R$, set $h_{(\alpha,\beta,r)}$ to be the transformation
	\begin{align*}
		h_{(\alpha,\beta,r)}:\R[t][x,y] &\to \R[t][x,y] \label{ecut2} \\
		f_t(x,y)&\mapsto h_{(\alpha,\beta,r)}(f_t)(x,y)=t^r f_{t}\circ \rho_{t} ^{(\alpha,\beta)}(x,y).
	\end{align*}
	Note that $h_{(\alpha,\beta,r)}$ can be obtained by extending the transformation $\widetilde{h}_{(\alpha,\beta,r)}$ in \eqref{homotenx,y} to polynomials $f_t\in \R[t][x,y]$ 
	and composing with $\varrho: \R[s,t][x,y] \to \R[t][x,y]$, $P(s,t,x,y)\mapsto P(t,t,x,y)$. 

Set $\varphi^{(\alpha,\beta,r)}_t$ to be the transformation given by
	\begin{align*}
			\varphi^{(\alpha,\beta,r)}_t: \Gamma_{f_t} & \to \Gamma_{h_{(\alpha,\beta,r)}(f_t)}\\\nonumber
			(x,y,f_{t}(x,y))& \mapsto \left(t^{ -\alpha}x,t^{ -\beta}y,h_{(\alpha,\beta,r)}(f_t)\circ \rho_{t}^{(-\alpha,-\beta)}(x,y)\right).
	\end{align*} Note that $\varphi^{(\alpha,\beta,r)}_t$ can be obtained by extending the transformation $\widetilde{\varphi}^{(\alpha,\beta,r)}_s$ in \eqref{eq:2} to polynomials $f_t\in \R[t][x,y]$ 
	and composing with $\varrho: \R[s,t][x,y] \to \R[t][x,y]$, $P(s,t,x,y)\mapsto P(t,t,x,y)$. 

 \begin{prop} \label{Prop:BiyecciondeSPPsbajohomoteciadeft}
	Let $f\in \R[x,y]$ be a polynomial with non-singular Hessian curve and let $f_t\in \R[t][x,y]$ be a perturbation of $f$. Then, for any bounded region $\Omega \subset \R^{2}$ there exist $\delta>0$ such that for $0<|t|<\delta$, $\varphi^{(\alpha,\beta,r)}_t$ is a one-to-one correspondence between 
\begin{enumerate}
\item[i)] $SPP(f_t)\cap \pi^{-1}(\Omega)$ and $SPP(h_{(\alpha,\beta,r)}(f_t))\cap \pi^{-1}\left(\rho_{t}^{(-\alpha,-\beta)}(\Omega)\right)$,
\item[ii)] $SPP(f_t)^*\cap \pi^{-1}(\Omega)$ and $SPP(h_{(\alpha,\beta,r)}(f_t))^*\cap \pi^{-1}\left(\rho_{t}^{(-\alpha,-\beta)}(\Omega)\right)$,
\item[iii)] $TSPP(f_t)\cap \pi^{-1}(\Omega)$ and $TSPP(h_{(\alpha,\beta,r)}(f_t))\cap \pi^{-1}\left(\rho_{t}^{(-\alpha,-\beta)}(\Omega)\right)$, \,\mbox{and}
\item[iv)] $TSPP(f_t)^*\cap \pi^{-1}(\Omega)$ and $TSPP(h_{(\alpha,\beta,r)}(f_t))^*\cap \pi^{-1}\left(\rho_{t}^{(-\alpha,-\beta)}(\Omega)\right)$
\end{enumerate}
for any $(\alpha,\beta,r)\in \Z^{3}$.
\end{prop}
\begin{proof} Since $H_{f}$ is non-singular in the closure of $\Omega$, by Corollary \ref{HFt non-sing}, there exists $\delta>0$ such that for $0<|t|<\delta$, $H_{f_{t}}$ has no singularities inside $\Omega$. By Proposition \ref{Prop:BiyecciondeSPPsbajohomotecia}, making $s=t$, we have i) and ii). And, by Proposition \ref{Prop:BiyecciondeTSPPsbajohomotecia}, we have iii) and iv).
\end{proof}

\section{Viro's Theorem for transversal special parabolic points}\label{sec:Viro'sthmforspps}

In this section we will adapt Viro's patchworking technique to the study of transversal special parabolic points on the graphs of polynomials. 

\begin{prop}\label{prop:lambdaaZ}
	Given a convex polyhedral subdivision $\tau$ of $\Delta\subset \R^2$ induced by $\lambda :\Delta \to \R_{\geq 0}$, there exists $d\in \Z_{>0}$ 
		such that $d\cdot \lambda(\Delta\cap \Z^2)\subset \Z$.
\end{prop}
\begin{proof}
	Take $\alpha\in\Delta\cap \Z^2$. If $\alpha$ is a $0$-dimensional polyhedron in $\tau$, then $\lambda(\alpha)\in \Z$. Otherwise, let $F\in\tau$ be the polyhedron with vertices $V_1,\dots,V_s$ such that $\alpha\in F$. There exist rational numbers $\alpha_i\in \Q$ such that $\displaystyle \alpha=\sum_{i=1}^s\alpha_i V_i$ and thus $\displaystyle \lambda(\alpha)=\sum_{i=1}^s\alpha_i \lambda(V_i)\in \Q$. The result follows from the fact that $\lambda(\Delta\cap \Z^2)$ is a finite set.
\end{proof}	
	
\begin{lemma}
	Let $\Delta \subset \R^2$ be a polyhedron and let $\tau$ be the polyhedral subdivision of $\Delta$ induced by the convex function $\lambda :\Delta \to \R_{\geq 0}$. If $\lambda(\Delta\cap \Z^2) \subset \Z$, then for $\widetilde{E}\in T(\lambda)$ the only vector of the form $(\alpha, \beta, 1)\in \R^3$ that is orthogonal to $\widetilde{E}$ has integer coordinates.
\end{lemma}
\begin{proof}
	Let $\,(x_1,y_1,\lambda(x_1,y_1)),\, (x_2,y_2,\lambda(x_2,y_2)),\, (x_3,y_3,\lambda(x_3,y_3))\in \widetilde{E}\cap \Z^3$ be points with the property that there are no points with integer coordinates inside the triangle $\Delta(v_1, v_2, v_3)$ with vertices $v_1=(x_1,y_1),\,v_2= (x_2,y_2),\,v_3= (x_3,y_3)\in \Z^2$, that is, \vspace{-0.2cm}
	\begin{equation}\label{AreaTriangulo=1}
		1=\area (\Delta(v_1, v_2, v_3))=|(v_2-v_1) \cdot(v_3-v_1)|.
	\end{equation}
	A vector $(\alpha, \beta, 1)\in \R^3$ is orthogonal to $\widetilde{E}$ if $(\alpha,\beta,1)\cdot(x_k,y_k,\lambda(x_k,y_k))$ is constant for $k=1, 2, 3$. The system of equations in the real variables $\alpha,\beta$ and $r$,
	\begin{equation}\label{Ecsparaaplanarcara}
	\begin{pmatrix}
		1&0& 0\\
		0 & 1 & 0\\
		\alpha & \beta & 1
	\end{pmatrix} \begin{pmatrix}
	x_k\\
	y_k\\
	\lambda(x_k,y_k)
	\end{pmatrix}=\begin{pmatrix}
	x_k\\
	y_k\\
	r
	\end{pmatrix} \quad k=1,2,3,\end{equation}
	sending $\widetilde{E}$ to the horizontal plane $\{Z=r\}$ is equivalent to the system of three equations
	$\displaystyle
	\begin{pmatrix} 
	x_1& y_1& 1\\
	x_2 & y_2 & 1\\
	x_3 & y_3 & 1
	\end{pmatrix} \begin{pmatrix}
	\alpha\\
	\beta\\
	-r 
	\end{pmatrix}=\begin{pmatrix}
	-\lambda(x_1,y_1)\\
	-\lambda(x_2,y_2)\\
	-\lambda(x_3,y_3)
	\end{pmatrix}$. The determinant of the $3\times 3$ matrix involved, given by $(x_2y_3-y_2x_3)-(x_1y_3-x_3y_1)+(x_1y_2-x_2y_1)=(x_3-x_1,y_3-y_1)\cdot( -(y_2-y_1),x_2-x_1)=(v_3-v_1) \cdot (v_2-v_1)^{\perp}$, is equal to $\pm 1$ by \eqref{AreaTriangulo=1}, 
	so the existence of integer solutions $\alpha, \beta, r\in \Z$ to the system \eqref{Ecsparaaplanarcara} is granted. 
\end{proof}	 

\begin{lemma}\label{Flatteningaface}
	Let $\Delta \subset \R^2$ be a polyhedron and let $\tau$ be the polyhedral subdivision of $\Delta$ induced by the convex function $\lambda :\Delta \to \R_{\geq 0}$. If $\lambda(\Delta\cap \Z^2) \subset \Z$ and $(\alpha, \beta, 1) \in \Z^3$ is orthogonal to $\widetilde{E}\in T(\lambda)$, then the linear transformation $l_{(\alpha,\beta)}:\R^3\to \R^3$ defined by the matrix $\begin{pmatrix}
		1&0& 0\\
		0 & 1 & 0\\
		\alpha & \beta & 1
		\end{pmatrix}\!$, satisfies $l_{(\alpha,\beta)}(\widetilde{E})\subset \{Z=r\}$ where $r:=\min\{(\alpha, \beta, 1)\cdot v;v\in T(\lambda)\}$ is an integer number.
\end{lemma}	
\begin{proof}	 
	The linear transformation $l_{(\alpha,\beta)}$ sends the face $\widetilde{E}\in T(\lambda)$ to a face in $T(\widetilde{\lambda})$ contained in the horizontal plane $\{Z=r\}$, where $\widetilde{\lambda}(i,j)=\lambda(i,j)+i\alpha+ j\beta$; while the remaining 2-dimensional faces in $T(\lambda)$ are sent to faces in $T(\widetilde{\lambda})$ above this horizontal plane, thus $r=\min\{(\alpha, \beta, 1)\cdot v;v\in T(\lambda)\}\in\Z$.
\end{proof}

From now on, and without loss of generality, by Proposition \ref{prop:lambdaaZ}, all our convex polyhedral subdivisions will be induced by convex functions sending points with integer coordinates to integer values.

Let $\tau$ be a convex polyhedral subdivision of $\Delta\subset \R^{2}$ induced by $\lambda: \Delta \to \R_{\geq 0}$. Let $f\in \R[x,y]$ be a polynomial with support in $\Delta$. Given $(\alpha,\beta)\in \Z^{2}$ the mapping 
\begin{equation*}
	\widetilde{\lambda}: (i,j) \mapsto \lambda(i,j)+\alpha i+\beta j
\end{equation*}
is also a convex function inducing $\tau$. Let $f_{t}$ be the patchworking polynomial of $f$ induced by $\lambda$ and let $\widetilde{f}_{t}$ be the patchworking polynomial induced by $\widetilde{\lambda}$, then 
\begin{equation}\label{eq:111}
\widetilde{f}_{t}=h_{(\alpha,\beta,0)}(f_{t})=t^{r}f_{t}(t^{\alpha}x,t^{\beta}y)
\end{equation} for some integer value $r\in \Z$. 

\begin{prop}\label{ortogonal}
	Let $\tau$ be a convex polyhedral subdivision of $\Delta \subset \R^2$ induced by $\lambda :\Delta \to \R_{\geq 0}$. Let $f \in\R[x,y]$ be a polynomial with support in $\Delta$ and let $f_t$ be the patchworking polynomial of $f$ induced by $\lambda$. If the vector $(\alpha,\beta,1) \in\Z^3$ is orthogonal to $\widetilde{E}\in T(\lambda)$, then the patchworking polynomial $\widetilde{f_t}$ induced by $\widetilde{\lambda}(i,j)=\lambda(i,j)+i\alpha+j\beta$ satisfies $\displaystyle h_{(\alpha,\beta,0)}\left(f_t|_{\widetilde{E}}\right)=\widetilde{f}_t^{[r]}$ for some constant $r\in\Z$.
\end{prop}
	\begin{proof} This result is direct consequence of (\ref{eq:111}) and Lemma \ref{Flatteningaface}.
\end{proof}
	
	\begin{thm}\label{Thm:Flatteningafacetozero}
		Let $\Delta \subset \R^2$ be a polyhedron with vertices in $\Z^2$ and let $\tau$ be a convex polyhedral subdivision of $\Delta$ induced by $\lambda :\Delta \to \R_{\geq 0}$. Let $f \in\R[x,y]$ be a polynomial with support in $\Delta$ and let $f_t$ be the patchworking polynomial of $f$ induced by $\lambda$. If the vector $(\alpha,\beta,1)\in\Z^3$ is orthogonal to $\widetilde{E}\in T(\lambda)$ and $\gamma=-\min \{(\alpha, \beta, 1)\cdot \upsilon; \upsilon \in T(\lambda)\}$ 
		then $\displaystyle h_{(\alpha,\beta,\gamma)}(f_t)$	is a perturbation of $f|_{\pi(\widetilde{E})}$.
	\end{thm}
	\begin{proof}
		Set $\displaystyle f(x,y)=\sum_{(i,j)\in \Delta}a_{i,j}x^{i}y^{j}\in \R[x,y]$. By Proposition \ref{ortogonal}, the patchworking polynomial $\widetilde{f_t}$ induced by $\widetilde{\lambda}(i,j)=\lambda(i,j)+i\alpha+j\beta$ satisfies $h_{(\alpha,\beta,0)}(f_t|_{\widetilde{E}})=\widetilde{f}_t^{[r]}$ for $r=\min \{(\alpha, \beta, 1)\cdot \upsilon; \upsilon \in T(\lambda)\}\in \Z$.
		The polynomial $\displaystyle h_{(\alpha,\beta,-r)}(f_t)(x,y)=t^{-r}\widetilde{f_t}(t^{\alpha}x,t^{\beta}y)=\sum_{(i,j)\in \Delta}a_{i,j}t^{\lambda(i,j)+i\alpha+j\beta-r}x^{i}y^{j}\in\R[t][x,y]$ is the patchworking polynomial $\widehat{f_t}$ induced by the convex function $\widehat{\lambda} :\Delta \to \R_{\geq 0}$, $\widehat{\lambda}(i,j):=\lambda(i,j)+i\alpha+j\beta-r$. In particular, $h_{(\alpha,\beta,-r)}(f_t|_{\widetilde{E}})=\widehat{f}_t^{\,[0]}$.	
		Expressing $\widehat{f}_t$ as the finite sum of level sets $\widehat{\lambda}^{-1}(r_0),\dots\widehat{\lambda}^{-1}(r_m)$ of $\widehat{\lambda}$ corresponding to the values
		$0=r_{0}<r_{1}<\cdots <r_{m}$, 
		we have that $\displaystyle \widehat{f}_t(x,y)= f|_{\pi(\widetilde{E})}+t \varphi_t,\; \mbox{with} \; \varphi_t(x,y)=\sum_{l=1}^m a_{i,j}t^{r_l-1}\widehat{f}_t^{\,[r_{l}]}(x,y)\in \R[t][x,y]$, as wanted.
	\end{proof}

Let $\Delta \subset \R^2$ be a polyhedron with vertices in the integer lattice $\Z^2$ and let $f\in\R[x,y]$ be a polynomial with support in $\Delta$. Let $\tau$ be the convex polyhedral subdivision of $\Delta$ induced by $\lambda :\Delta \to \R_{\geq 0}$. Denote by $TSPP(f,\lambda)^{*}$	the set of pairs
\begin{equation*}
 	TSPP(f,\tau)^{*}:=\bigcup_{E\in \tau}\{(E,p); p\in TSSP(f|_E)^{*} \}.
\end{equation*}

\begin{thm}\label{Viro'sThmfortspps}(Viro's Theorem for transversal special parabolic points)
	Let $\Delta \subset \R^2$ be a polyhedron with vertices in $\Z^2$ and let $\tau$ be the convex polyhedral subdivision of $\Delta$ induced by $\lambda :\Delta \to \R_{\geq 0}$. Let $f \in\R[x,y]$ be a polynomial with non-singular Hessian curve, and support in $\Delta$. If $f_t$ is the patchworking polynomial of $f$ induced by $\lambda$, then there exists $\delta>0$ such that for $0<|t|<\delta$, there is an inclusion
		\begin{equation*}
		\varphi_t:TSPP(f,\tau)^{*}\hookrightarrow TSPP(f_t)^{*}.
		\end{equation*}
\end{thm}
\begin{proof}\hspace{-5pt}Let $\textstyle (\alpha,\beta,1)\in\Z^3$ be an orthogonal vector to the face $\textstyle \widetilde{E}\in T(\lambda)$ and let $\textstyle r:=-\min \{(\alpha, \beta, 1)\cdot \upsilon; \upsilon \in T(\lambda)\}\in \Z$. Then, by Theorem \ref{Thm:Flatteningafacetozero}, the polynomial $h_{(\alpha,\beta,r)}(f_t)(x,y)$ defines a perturbation of $f|_{E}$, where $E=\pi(\widetilde{E})$ is in $\tau$. 
	
	For each face $E$ in $\tau$, set $\mathit{C}_E:=\pi(TSPP(f|_E)^*)$. Choose $\varepsilon>0$ small enough such that, for any $q,q^{\prime}\in \mathit{C}_E$, we have $D(q,\varepsilon),D(q^{\prime},\varepsilon)\subset (\R^{*})^{2}$ and $D(q,\varepsilon)\cap D(q^{\prime},\varepsilon)\neq \emptyset$. By Corollary \ref{Cor:InclusionTSPPdefenTSPPdeFt}, there exists $\delta_1>0$ such that for
	$0<|t|<\delta_1$ there is an inclusion
	\begin{equation*}
		\psi_{t}^{E}:TSPP(f|_E)^{*}\hookrightarrow TSPP(h_{(\alpha,\beta,r)}(f_t))^{*}
	\end{equation*}
	satisfying $\pi(\psi_t^{E}(p))\in D(\pi(p),\varepsilon)$. Thus, for each $q\in \mathit{C}_E$, $\pi\left(TSPP(h_{(\alpha,\beta,r)}(f_t)\right)^{*} \cap D(q,\varepsilon)\neq \emptyset$. By Proposition \ref{Prop:BiyecciondeSPPsbajohomoteciadeft}, there exists $\delta_2>0$ such that for $0< |t|<\delta_2$, 
	\begin{equation*}
		\left(\varphi^{(\alpha,\beta,r)}_t\right)^{-1}: TSPP(h_{(\alpha,\beta,r)}(f_t))^{*} \to TSPP(f_t)^{*} 
	\end{equation*}	
	is a bijection. Hence, for $0<|t|<\min\{\delta_1,\delta_2\}$, we have the inclusion
	\begin{equation*}
		\varphi_t^E: TSPP(f|_E)^{*}\stackrel{\psi_t^E}{\hookrightarrow}
		 TSPP(h_{(\alpha,\beta,r)}(f_t))^{*}\stackrel{\stackrel{\left(\varphi^{(\alpha,\beta,r)}_t\right)^{-1}}{\simeq}}{\longrightarrow} TSPP(f_t)^{*}.
	\end{equation*}
	
	By Proposition \ref{Prop:phitkeepspointsawayV2}, there exists $\delta_{3}>0$, such that, for $|t|<\delta_{3}$, $\mathrm{Im}(\varphi_{t}^E)\cap \mathrm{Im}(\varphi_{t}^{E^{'}})= \emptyset$ for $E\neq E^{'}$. Therefore, for $0<|t|<\delta:=\min\{\delta_1,\delta_2,\delta_3\}$, we obtain an inclusion 	\begin{equation*}
		\varphi_t:TSPP(f,\tau)^{*}\hookrightarrow TSPP(f_t)^{*}.
	\end{equation*}
\end{proof}

\begin{cor}\label{Cor:lemaViro}
	Let $\Delta \subset \R^2$ be a polyhedron with vertices in $\Z^2$ and let $\tau$ be the convex polyhedral subdivision of $\Delta$ induced by $\lambda :\Delta \to \R_{\geq 0}$. Let $f \in\R[x,y]$ be a polynomial with support in $\Delta$. If $f_t$ is the patchworking polynomial of $f$ induced by $\lambda$, then there exists $\delta>0$ such that for $0<|t|<\delta$, we have
\begin{equation*}
		|TSPP(f,\tau)^{*}|\leq |TSPP(f_t)^{*}|.
		\end{equation*}
\end{cor}
\begin{proof}
It follows from Theorem \ref{Viro'sThmfortspps}.
\end{proof}
	
\section{An Application of Corollary \ref{Cor:lemaViro}}\label{sec:Application}

In this section, we give an example to show how to use Corollary \ref{Cor:lemaViro} to build families of polynomials with a prescribed number of transversal especial parabolic points.

Let $\Delta\{a,b,c\} \subset \R^2$ denote the triangle with vertices $a$, $b$, $c\in \R^2$. We will denote by 
\begin{itemize}
	\item $\Delta_d:=\Delta\{(2,2), (d-2,2), (2,d-2)\}$,
	\item $\Delta^1_{(i,2)}:=\Delta\{(i, 2), (i+1,2), (i, 3)\}$,
	\item $\Delta^2_{(k,l)}:=\Delta\{(k+1, l), (k, l+1), (k, l+2)\}$, and 
	\item $\Delta^3_{(k,l)}:=\Delta\{(k+1, l), (k, l+2), (k+1, l+1)\}$
\end{itemize}for $d, i,k,l\in \Z_{\geq 2}$.

		

\begin{figure}[H]

\begin{picture}(200,160)

\put(0,0){\vector(0,1){150}}\put(0,0){\vector(1,0){150}}

{\color{magenta}
	\put(20,20){\line(0,1){120}}\put(20,20){\line(1,0){120}}
	\put(20,140){\line(1,-1){120}}
	\put(20,120){\line(1,-1){100}}
	\put(20,100){\line(1,-1){80}}
	\put(20,80){\line(1,-1){60}}
	\put(20,60){\line(1,-1){40}}
	\put(20,40){\line(1,-1){20}}
	\put(40,20){\line(0,1){100}}
	\put(60,20){\line(0,1){80}}
	\put(80,20){\line(0,1){60}}
	\put(100,20){\line(0,1){40}}
	\put(120,20){\line(0,1){20}}
	\put(20,140){\line(1,-2){60}}
	\put(20,120){\line(1,-2){40}}
	\put(20,100){\line(1,-2){40}}
	\put(20,80){\line(1,-2){20}}
	\put(20,60){\line(1,-2){20}}
	\put(40,120){\line(1,-2){40}}
	\put(60,100){\line(1,-2){40}}
	\put(80,80){\line(1,-2){20}}
	\put(100,60){\line(1,-2){20}}
}
\put(14,10){\makebox(0,0){$(2,2)$}}

{\color{red}
	\put(200,110){\line(0,1){20}}
	\put(180,150){\line(1,-1){20}}
	\put(180,150){\line(1,-2){20}}}
\put(260,130){\makebox(100,0){\text{Triangle $\Delta^3_{(k,l)}$}}}
\put(230,130){\vector(-1,0){15}}

{\color{cyan}
	\put(210,86){\line(0,1){20}} 
	\put(210,86){\line(1,-1){20}}
	\put(210,106){\line(1,-2){20}}}
\put(260,86){\makebox(100,0){\text{Triangle $\Delta^2_{(k,l)}$}}}
\put(250,86){\vector(-2,0){15}}

{\color{blue}
	\put(190,30){\line(0,1){20}} 
	\put(190,50){\line(1,-1){20}}
	\put(190,30){\line(1,0){20}}}
\put(260,40){\makebox(100,0){\text{Triangle $\Delta^1_{(i,2)}$}}}
\put(233,40){\vector(-1,0){15}}

\end{picture} 

\end{figure}

	
	
	
The triangular subdivision $\tau$ of $\Delta_d$ obtained by dividing $\Delta_d$ into $\Delta^1_{(i,2)}, \Delta^2_{(k,l)}, \Delta^3_{(k,l)}$ with $i\in\{2, 3, \dots, d\}$, $k\in\{2, 3, \dots, d-2\}$ and $l\in\{2, \dots, d-k-2\}$, is convex. This subdivision is induced by a convex function $\lambda:\Delta_d \to \R_{\geq 0}$ that has been used in several works, for example in \cite{BertrandErwan} and \cite{Lucia}. Consider the polynomial types:
$P^1_{(i,2)}(x,y):=x^{i}y^{2}(1+x+y)$, $\, P^2_{(k,l)}(x,y):= x^{k}y^{l}(x+y+y^2)$ and 
$P^3_{(k,l)}(x,y):=x^{k}y^{l}(x+xy+y^2)$, whose support lies in  $\Delta^1_{(i,2)}, \Delta^2_{(k,l)}$ and $\Delta^3_{(k,l)}$, respectively.

\begin{thm}\label{thm:MainBound}
	\rm{Let $f=x^2y^2g(x,y)$ be the degree $d$ polynomial with support in the triangle $\Delta_d$, where $g \in \R[x,y]$ is a complete polynomial of degree $d-4$. Let $\tau$ be the polyhedral subdivision induced by $\lambda:\Delta_d \to \R_{\geq 0}$ as above, and let $f_t\in \R[x,y]$ be the patchworking polynomial of $f$ induced by $\lambda$. Then, for $d\leq 10,000$, there exists $\varepsilon>0$ such that \begin{equation*}
		|TSPP(f_t)^{*}|\geq(d-4)(2d-9)
		\end{equation*} for $0<|t|<\varepsilon$.}
\end{thm} 

\begin{proof}
	The convex polytope $T(\lambda)$ is the union of $d-2$ triangular faces of type $\widetilde{\Delta}^1_{(i,2)}$ and $\frac{(d-4)^2-(d-4)}{2}\!\!=\!\frac{(d-5)(d-4)}{2}$ faces of type $\widetilde{\Delta}^2_{(k,l)}$ and $\widetilde{\Delta}^3_{(k,l)}$, respectively, 
	whose proyections on the $xy$-plane are given by $\pi\left( \widetilde{\Delta}^1_{(i,2)} \right) = \Delta^1_{(i,2)}$, $\pi\left( \widetilde{\Delta}^2_{(k,l)}\right) = \Delta^2_{(k,l)}$ and $\pi\left( \widetilde{\Delta}^3_{(k,l)}\right) = \Delta^3_{(k,l)}$.

The restrictions of $f$ to the faces in $T(\lambda)$ are equal to	
\begin{equation*}
	f|_{\Delta^1_{(i,2)}} = P^1_{(i,2)}, \quad f|_{\Delta^2_{(k,l)}} = P^2_{(k, l)}\quad \mbox{and} \quad
f|_{\Delta^3_{(k,l)}} = P^3_{(k, l)}.
\end{equation*} 
Using computer software \textrm{Mathematica} \cite{Mathematica}, we can see that, given $d\leq 10000$, $|TSPP(P^1_{(i,2)})^{*}|=1$ for $i \geq 2$, $|TSPP(P^2_{(k,l)})^{*}|=1$ and $|TSPP(P^3_{(k,l)})^{*}|=3$ for $k\in\{2, 3, \dots, d-2\}$ and $l\in\{2, \dots, d-k-2\}$. 
Corollary \ref{Cor:lemaViro} guarantees that the number of transversal  special parabolic points in the graph of $f_t\in \R[t][x,y]$ satisfies
\begin{equation*}
|TSPP(f_t)^{*}|\geq 1(d-4)+4\left(\frac{(d-5)(d-4)}{2}\right)=(d-4)(2d-9), 
\end{equation*}for sufficiently small values of $t\neq 0$. 
\end{proof}


 Although we haven't found a proof of the fact that the inequality $|TSPP(P^3_{(k,l)})^{*}|\geq 3$ holds for $k,l\in\Z_{\geq 2}$, we firmly believe that the bound we give in the statement of Theorem \ref{thm:MainBound} works in general so the restriction on the degree can be removed from its hypotheses.


 

\bibliographystyle{plain} 

\end{document}